\documentclass[12pt]{amsart}
\usepackage{amssymb}
\usepackage{amsmath}
\usepackage{bbm} 
\usepackage{enumitem}
\usepackage{verbatim}
\usepackage{pictex}

\newcommand\AND{\quad\text{and}\quad}
\newcommand\bal{\mathsf{bal}}

\newcommand\C{\mathbb C}
\newcommand\df{\mathsf{df}}
\newcommand\Gc{\mathcal{G}}
\newcommand\modulo{\rm mod\ }

\newcommand\Prob{\mathbb{P}}
\newcommand\Q{\mathbb{Q}}
\newcommand\R{\mathbb R}
\newcommand\rd{\text{\rm rd}}
\newcommand\sgn{\mathsf{sgn}}
\newcommand\spec{\mathsf{spec}}
\newcommand\uno{\mathbf{1}}

\numberwithin{equation}{section}

\newtheoremstyle{mythm}
  {9pt}
  {9pt}
  {\itshape}
  {0pt}
  {\bfseries}
  {}
  { }
  {\thmnumber{(#2)}\thmname{ #1}\thmnote{ #3}}

\newtheoremstyle{mydef}
  {9pt}
  {9pt}
  {\normalfont}
  {0pt}
  {\bfseries}
  {}
  { }
  {\thmnumber{(#2)}\thmname{ #1}\thmnote{ #3}}

\theoremstyle{mythm}
\newtheorem{thm}[equation]{Theorem.}
\newtheorem{pro}[equation]{Proposition.}
\newtheorem{lem}[equation]{Lemma.}

\newtheorem{cor}[equation]{Corollary.}

\theoremstyle{mydef}
\newtheorem{dfn}[equation]{Definition.}
\newtheorem{exa}[equation]{Example.}

\newtheorem{rmk}[equation]{Remark.}

\begin{document}$\,$ \vspace{-1.2truecm}
\title{The spectra of graph substitutions}
\author{\bf Thomas Hirschler, Wolfgang Woess}
\address{\parbox{.8\linewidth}{Institut f\"ur Diskrete Mathematik,\\ 
Technische Universit\"at Graz,\\
Steyrergasse 30, A-8010 Graz, Austria\\}}
\email{\rm thirschler@tugraz.at, woess@tugraz.at}
\date{\today} 

\begin{abstract}
Let $(X,E_X)$ and $(V,E_V)$ be finite connected graphs without loops. 
We assume that $V$  has two distinguished vertices 
$a,b$ and an automorphism $\gamma$ which exchanges $a$ and~$b$. 
The $V$-edge substitution of $X$  is the graph $X[V]$
where each edge $[x,y] \in E_X$ is replaced by a copy of $V$, identifying $x$ with $a$ and $y$ with 
$b$ or vice versa. (The latter choice does not matter; it yields isomorphic graphs.) The aim is to 
describe the spectrum of $X[V]$ in terms of the spectra of $X$ and $V$. Instead of the spectra
of the adjacency matrices, we consider the versions which are normalised by dividing each row by the 
row sum (the vertex degree). These are stochastic, reversible matrices, and our approach applies
more generally to reversible transition matrices corresponding to arbitrary positive edge 
weights invariant under $\gamma$.
We write $P$ for the transition matrix over $X$ and $Q$ for the one over $V$. Together, they
induce the matrix $P_*$ over $X[V]$. 

The main part of the spectrum of $P_*$ is the response of the natural frequencies of 
$X$ to substituting $V$, given by a functional equation coming from a rational function 
induced by~$Q$. A second part comes from specific eigenvalues of $Q$, if present. 
Finally, there is the part of $\spec(P_*)$ having eigenfunctions with $X$ as a nodal set. 
The results depend on issues like whether
$X$ has circles of even length and on the eigenvalues of the restriction of $Q$
to $V \setminus \{ a,b\}$, which are classified into 4 possible types. 
Quite subtle is the issue of determining the multiplicities of the latter as eigenvalues
of $P_*$ in terms of the input. 

Strichartz~\cite{St} was the first to consider this setting and to 
explain the main mechanism how the spectrum changes under this construction. 
The present work completes his
considerations in somewhat different terms and greater detail.
\end{abstract}


\maketitle

\markboth{{\sf T. Hirschler and W. Woess}}
{{\sf Graph substitution}} 
\baselineskip 14.45pt

\vspace*{-.6cm}

\section{Introduction}\label{sec:intro}

Let $(X,E_X)$ -- in short, $X$ -- be a connected, locally finite graph with vertex set $X$
and edge set $E_X$ 
and $(V,E_V)$ a finite, connected graph, both without loops. 
We suppose that $V$ has two distinguished vertices $a, b$ and a graph automorphism $\gamma$ 
such that $\gamma a = b$ and $\gamma b = a$.

The $(V, a,b)$-edge substitution -- in short, $V$-edge substitution of $X$ -- is the graph $X[V]$
where each edge $[x,y] \in E_X$ is replaced by a copy of $V$, identifying $x$ with $a$ and $y$ with 
$b$ or vice versa. In view of the presence of the automorphism $\gamma$ the latter choice does
not matter, i.e., it yields isomorphic graphs. Figure 1 shows a simple example.
$$
\beginpicture 

\setcoordinatesystem units <4.2mm,4.2mm> 

\setplotarea x from -4 to 28, y from -5 to 4.5

\put {$X$} at -3 3.6
\plot 4 0  1.236 3.804  -3.236 2.351  -3.236 -2.351  1.236 -3.804  4 0 /
\multiput {$\scriptstyle \bullet$} at 
  4 0  1.236 3.804  -3.236 2.351  -3.236 -2.351  1.236 -3.804 /

\put {$V$} at 8.3 1.8
\plot 10.351 1.2  8 0  10.351 -1.2  12.702 0  10.351 1.2  10.351 -1.2 /
\multiput {$\scriptstyle \bullet$} at   8 0    12.702 0 /
\put {$a$} at 7.6 -0.4
\put {$b$} at 13,102 -0.4

\put{$\Longrightarrow$} at 16.5 0

\put {$X[V]$} at 19.5 3.6
\plot 19.564 0  20.764 -2.351  21.964 0  20.764 2.351  19.564 0  21.964 0 /
\plot 22.629 4.217  20.764 2.351  23.370 1.936  25.236 3.804  22.629 4.217  23.370 1.936 /
\plot 22.629 -4.217  20.764 -2.351  23.370 -1.936  25.236 -3.804  22.629 -4.217  23.370 -1.936 /
\plot 27.588 2.608  25.236 3.804  25.647 1.196  28 0  27.588 2.608  25.647 1.196 /
\plot 27.588 -2.608  25.236 -3.804  25.647 -1.196  28 0  27.588 -2.608  25.647 -1.196 /
\multiput {$\scriptstyle \bullet$} at 
  20.764 -2.351  20.764 2.351  25.236 -3.804  25.236 3.804  28 0 /

\put {Figure~1} at 12 -4.8
\endpicture
$$

\medskip

We want to understand the spectrum of $X[V]$. For a discussion in the context of graph
products, see the final parapgraph of this paper. 
The basic spectral theory of graphs mostly refers to the adjacency matrix or the 
Laplacian matrix, see
e.g. the books by {\sc Cvetkovi\'c, Doob and Sachs}~\cite{CDS}, {\sc Biggs}~\cite{Bi},
{\sc Chung}~\cite{Ch}, {\sc Godsil and Royle}~\cite{GR}.  As compared to the spectral 
theory of the ordinary adjacency matrix,  our approach is slightly different in 
two ways: first, we consider locally finite 
weighted graphs, compare e.g. with {\sc Colin de Verdi\`ere}~\cite{CdV}:
for a graph $X$ each non-oriented
edge $e=[x,y] \equiv [y,x] \in E_X$ of $X$ has a positive conductance 
$a_X(e) = a_X(x,y) = a_X(y,x)$, and $a_X(x,y)=0$ 
when $[x,y] \notin E_X$.
The basic assumption in the place of local finiteness is that the total conductance 
$m_X(x) = \sum_y a_X(x,y)$ is finite at every vertex $x \in X$. Instead of the 
weighted adjacency
matrix $A_X = \bigl(a_X(x,y)\bigr)_{x,y \in X}$  we consider the normalised version 
$P_X = P = \bigl(p(x,y)\bigr)_{x,y \in X}\,$, where $p(x,y)= a_X(x,y)/m_X(x)$. This is a 
stochastic matrix, often called reversible, which stands for the fact that multiplication 
by $m_X(\cdot)$ symmetrises $P$, that is, $m_X(x)p(x,y)=m_X(y)p(y,x)$.
Associated with it, there is the rich theory of electrical
networks and reversible
Markov chains, see e.g. {\sc Lyons and Peres}~\cite{LP} or {\sc Woess}~\cite[Ch. 4]{WMarkov}.
For the associated spectral theory of self-adjoint, non-negative matrices,  
{\sc Horn and Johnson}~\cite{HJ} is the basic general reference. When the  
conductances are constant 
($=1$), $P$ is the transition matrix of the so-called simple random walk (SRW) on 
the underlying graph. 

For the graph $(V,E_V)$ we use the following notation:
each edge $e = [u,v]$ of $V$ has a postive conductance
$a_V(e) = a_V(u,v) = a_V(v,u)$, with $a_V(u,v)=0$ when $[u,v]\notin E_V\,$,
and $m_V(u) = \sum_v a_V(u,v)$. By an automorphism of
a network we intend a conductance-preserving bijection $\gamma$ of the vertex set,
and we require existence of such an automorphism of $V$ which exchanges $a$ with $b$. In particular,
$m_V(a)=m_V(b)$.
The edge set of $X[V]$ is $E_{X[V]} = E_X \times E_V$ and the vertex set is 
$X \cup (E_X \times V^o)$,  where $V^o= V \setminus \{a,b\}$ (the interior of $V$). 

More formally this may be described as follows:
first, we make a choice for each $e=[x,y]\in E_X$ which of its endpoints should be identified
with $a$, resp. with $b$ when substituting $V$ for that edge. Thus, $e=[e^a\,,e^b]$ with
$\{e^a\,,e^b\} = \{x,y\}$. We may think of this as putting an orientation (arrow) on each edge, 
such that $e^a$ is the initial and $e^b$ the terminal vertex of $e$. None of our results
depends on the choice of this orientation. First we
consider the graph $E_X \times V$ consisting of disjoint copies $\{e\}\times (V,E_V)$  of $V$ indexed
by all $e \in E_X\,$, where of course the edges are $[(e,u),(e,v)]$ whenever $[u,v] \in E_V\,$. 
Then we define the identification map as the graph homomorphism
\begin{equation}\label{eq:ident}
\begin{aligned}
&\pi: E_X \times V \to X \cup (E_X \times V^o),\\
&\pi(e,a) = e^a\,, \; \pi(e,b) =e^b\,, \; \pi(e,v) = (e,v) \; \text{for}\; v \in V^o\,,
\end{aligned} 
\end{equation}
so that $\pi|_{\{e\}\times V}$ is a graph isomorphism for each $e \in E_X\,$.

The conductances on $E_{X[V]}$ are the products of the conductances of the factors,
that is, for an edge $(e, \tilde e) \in E_X \times E_V$, the conductance  is
\begin{equation}\label{eq:cond}
a_{X[V]}(e, \tilde e) = a_X(e)a_V(\tilde e)\,.
\end{equation}
Thus, the total conductances of the substituted network are
$$
\begin{aligned}
m_{X[V]}(x) &= m_X(x)m_V(a) \quad \text{for }\; x \in X\,, \AND\\
m_{X[V]}(e,v) &= a_X(e)m_V(v) \quad \text{for }\; (e,v) \in E_X \times V^o\,. 
\end{aligned}
$$
For any $e = [x,y] \in E_X$ with $x=e^a$ and for $u,v \in V^o$, we get
$$
\begin{gathered}
p_*\bigl(x,(e,v)\bigr) = p(x,y)q(a,v)\,,\quad p_*\bigl(y,(e,v)\bigr) = p(y,x)q(b,v)\\
p_*\bigl((e,u),(e,v)\bigr) = q(u,v)\,,\quad   
p_*\bigl((e,u),x)\bigr) = q(u,a)\,, \AND  p_*\bigl((e,u),y)\bigr) = q(u,b).
\end{gathered}
$$
When $x = e^b$ one has to exchange $a \leftrightarrow b$.

The stochastic transition matrix $P=P_X$ associated with the network $(X, a_X)$ as above
governs a Markov chain (random walk) 
$(Z_n)_{n\ge 0}$ on $X$ with $\Prob_{x_0}[Z_{n+1}=y \mid Z_n = x]$.\footnote{We shall use 
Markov chain terminology on several occasions, but for the reader who wants to avoid 
probability, things can be formulated in terms of combinatorics of paths, see 
the final part of \cite[\S 1.D]{WMarkov}}
The index $x_0$ indicates the starting point $x_0\,$. 

Let $\ell_2(X,m_X)$ be the Hilbert
space of all functions $f: X \to \C$ with $(f,f)_X < \infty\,$, where the inner product is 
$$
(f,g)_X = \sum_{x \in X} f(x) \overline{g(x)}\, m_X(x).
$$
$P$ acts on this space as a self-adjoint operator by
$$
Pf(x) = \sum_y p(x,y)f(y).
$$
We let $\spec(P)$ be the associated spectrum.
Analogously, we have the transition matrix associated with $(V,a_V)$ for which we 
write 
$$
Q = \bigl(q(u,v)\bigr)_{u,v \in V}\,, \quad \text{where} \quad q(u,v) = \frac{a_V(u,v)}{m_V(u)}
$$
and the associated Markov chain $(Y_n)_{n\ge 0}$ on the state space $V$.

Finally, we also have the transition matrix $P_*$ of the network $(X[V], a_{X[V]})$. 
It governs  the
Markov chain $(Z_n^*)_{n \ge 0}$ with state space $X[V]$, and it acts as a self-adjoint
operator on $\ell^2(X[V], m_{X[V]})$ with inner product defined as above.

\smallskip

\textbf{The goal is to study how $\spec(P_*)$ is obtained from 
$P$ and $Q$.}

\smallskip

This construction and the issue of describing the spectrum were 
introduced and studied by {\sc Strichartz}~\cite{St} in the equivalent 
setting where the \emph{Laplacian} $P-I$ is studied. ($I= I_X$ is the identity matrix over $X$.) 
When embarking on the present work,  we were not aware 
of this work: we are grateful to a referee for pointing out its 
existence. It turns out 
that the discussion in \cite{St} is not complete in several points which we shall clarify 
at the repsective places. To our knowledge, the only other source where a spectral property 
of edge substitution was studied is  Thm. 9.89 in the second author's 
book~\cite{WMarkov} where  $X$ is infinite, $V$ is a finite line 
graph and the spectral radius is considered. As a mattter, this widely unobserved 
original result was the main source for the ideas presented here. 

There is no nicely closed 
formula which describes the desired result, but several pieces have to be carefully put 
together which  depend on structural 
features of $X$ such as bipartiteness 
and cycle structure as well as on  properties of the eigenfunctions (= right eigenvectors) of 
$Q$ and in particular of its restriction $Q_{V^o}$ to~$V^o$. 

\smallskip

In the present paper, we assume from \S \ref{sec:Green} onwards that both $X$ and 
$V$ are finite. We also assume that the subgraph $V \setminus \{b\}$ of $V$ 
(and thus also $V \setminus \{a\}$) is 
connected. 

We start in \S \ref{sec:Green} with general 
facts about the Green function of a reversible stochastic matrix, that is, the 
resolvent matrix.
Typical equations for the latter as well as related functions can be obtained in 
terms of walk generating functions in the complex variable $1/z$, valid a priori 
in the domain of 
convergence of these power series from where they extend analytically as rational 
functions. This approach,  in somewhat different terms than \cite{St}, 
allows for rather fine computational \& proof details. 

A crucial role is played by $\varphi(z)$, the reciprocal of the ``hitting function'' of 
$b$ when the Markov  chain on $V$ governed by $Q$ starts at $a$. Along with two further 
specific functions, it is the key to a boundary value problem on the graph $V$ with 
boundary $\{ a,b \}$, which appears as a main part in Strichartz' work 
\cite{St}, explaining the principal mechanism how the spectrum changes under substitution. 
This leads to  the first important part of the spectrum of $P_*$ via 
solutions $\lambda^*$ of the equations $\varphi(\lambda^*)=\lambda$, 
where $\lambda \in \spec(P)$ -- the response of the natural 
frequencies of $X$ to substituting~$V$. 

In \S \ref{sec:into}, we classify the eigenvalues of $Q$ as well as of $Q_{V^o}$ into 4 types,
respectively, depending on the corresponding eigenspaces. According to the type, the eigenfunctions 
may (or may not) be extended to eigenfunctions of $P_*\,$. The more important 
ones are those of $Q_{V^o}$ -- the \emph{Dirichlet eigenfunctions} of $Q$ in the 
terminology of \cite{St}, compare e.g. with {\sc Shi and Yu}~\cite[\S 2]{SY}. They lead to 
eigenfunctions that are $\equiv 0$ on the set of
vertices $X \subset X[V]$, that is, those which have $X$ as a nodal set. As 
already observed in \cite{St}, the cycle structure
of $X$ plays an important role here. In terms of the latter 
we are able to provide detailed constructions of sets of linearly independent eigenfunctions
which may allow for algorithmic implementation. 

In \S \ref{sec:complete}, we show that the extension constructions of \S \ref{sec:into} yield all 
possible eigenfunctions which are nodal ($\equiv 0$) on $X$. This is important for understanding
the multiplicities of the corresponding eigenvalues of $P_*$ in terms of the input from 
$\spec(Q_{V^o})$ and $\spec(Q)$. An analogue of this is not present in 
the discussion of \cite{St}, where in fact a part of the nodal eigenfunctions of 
\S \ref{sec:into} as well as the possible contributions from  $\spec(Q)$
remained unobserved. 

We then summarise in \S \ref{sec:spec}, describing the entire spectrum of $P_*$ and 
the multiplicities of the eigenvalues. The latter are subtle and depend on the types of 
the eigenvalues, and an exceptional set of candidate eigenvalues has to be excluded when
$X$ is a tree or posseses just one cycle which has odd lentgh. We also show 
that the spectral gap, or in other words, $1-\lambda_1^*\,$, where 
$\lambda_1^*$ is the second largest eigenvalue 
 of $P_*\,$, cannot be nodal and must always come from the equation 
$\varphi(\lambda_1^*)=\lambda_1\,$, when $V^o$ is connected. 

In \S \ref{sec:examples}, we present a few basic examples along with the necessary
computational details.

The final, short \S \ref{sec:outlook} contains a brief discussion on spectrally
(non-)dominant vertices, an outlook on infinite graphs
including self-similarity, as well as an outlook on substituting a  manifold, and concludes
with additional references  concerning the affinity of graph substitution with graph products.

\smallskip
\noindent
\textbf{Acknowledgement.} We thank the two referees for helpful hints and suggestions, in particular
one of them for pointing out the paper of Strichartz and the other one for insisting on a 
sloppy side remark concerning adjancency matrices.

\section{Green kernels and a functional equation for eigenvalues}\label{sec:Green}

Before starting, we provide a short compendium on part of our eigenvalue notation.

\begin{itemize}
 \item Eigenvalues of $P$ are denoted $\lambda$.  
       Corresponding multiplicity: $\nu_P(\lambda)$.
 \item Eigenvalues of $P_*$ are denoted $\lambda^*$. 
       Corresponding multiplicity: $\nu_*(\lambda)$.
 \item Eigenvalues of $Q$ are denoted $\lambda$, resp. $\lambda^*$
 when embedded into $\spec(P_*)$.\\
      Corresponding multiplicity in $\spec(Q)$: $\;\nu_Q(\lambda)$,
      resp. $\nu_Q(\lambda^*)$.       
\item Eigenvalues of $Q_{V^o}$ (restriction ov $Q$ to $V^o$)
  are denoted $\lambda$, resp. $\lambda^*$
 when embedded into $\spec(P_*)$. 
      Corresponding multiplicity in $\spec(Q_{V^o})$: $\;\nu_o(\lambda)$,
      resp. $\nu_o(\lambda^*)$.       
\end{itemize}

The spectrum of each of our matrices is contained in the real interval $[-1\,,\,1]$.
What follows is formulated in terms of $P$ but applies as well to $Q$ and $P_*$ and
submatrices obtained by restricting them to subsets of the respective vertex sets.
We start with a little general lemma that will be used later.

\begin{lem}\label{lem:lambda}
For every $\lambda \in  [-1\,,\,1]$ there are a bipartite graph $(X,E_X)$ and an associated 
reversible transition matrix $P$ such that $\lambda \in \spec(P)$. 
\end{lem}

\begin{proof} The graph is the $2N$-circle with vertices $x_0\,, \dots, x_{2N-1}$ and edges 
$[x_{i}\,,x_{i+1}]$ with indices modulo $2N$, where $N\ge 2$. For a parameter $a \in [0\,,\,1]$, 
the edge conductances are $a(x_{i}\,,x_{i+1}) =1$ for $i=0\,\dots, 2N-2$ and $a(x_{2N-1}\,,x_0)=a$.
We write $P_{a,2N}$ for the resulting transition matrix, and we want to locate its second largest
eigenvalue $\lambda_1$ for $0 \le a \le 1$ (the largest one is $\lambda_0=1$). 
In the standard case $a = 1$ we get the simple random walk on the
$2N$-circle, and it is well-known that $\lambda_1 = \cos \frac{\pi}{N}$. 
If $a=0$ then this means that the edge $[x_{2N-1}\,,x_0]$ is deleted, and the 
graph is reduced to the path with $2N$ vertices, while $P_{0,2N}$ is simple random walk on 
that graph with reflecting
end-vertices, compare with \cite[Example 5.6]{WMarkov}. In this case, $\lambda_1 = \cos\frac{\pi}{2N-1}$. 
By continuity, for every $\lambda \in J_N= \bigl[\cos \frac{\pi}{N}\,,\,\cos \frac{\pi}{2N-1}\bigr]$ 
there is $a \in [0\,,\,1]$ such that $\lambda_1(P_{a,2N}) = \lambda$. Now 
$\bigcup_{N=2}^{\infty} J_{N} = [0\,,\,1)$. Therefore each $\lambda \in [0\,,\,1)$ appears as 
$\lambda_2(P_{a,2N})$ for some $N$ and $a$. Of course, $\lambda=1$ is always an eigenvalue, and since 
the $2N$-circle is bipartite, the spectrum is symmetric, concluding the proof.
\end{proof}

For  $z \in \C \setminus \spec(P)$, the inverse
$$
\Gc(z) = (z\, I - P)^{-1}
$$
is well defined as the resolvent of the operator $P$ on $\ell_2(X,m_X)$.
It can be viewed as a matrix over $X$, whose $(x,y)$-matrix element is denoted
$G(x,y|z) = \Gc(z)\uno_y(x)$ and often called the Green function or Green kernel of $P$.
 The norm = spectral radius of the operator $P$ is
$$
\lambda_0(P) = \limsup_{n \to \infty} p^{(n)}(x,y)^{1/n}\,, 
$$
where for $p^{(n)}(x,y) = \Prob_x[Z_n=y]$ is the $(x,y)$-element of the 
$n^{\textrm{th}}$ matrix power $P^n$. This is the largest 
eigenvalue of $P$, often also denoted $\rho(P)$. Since $X$ is assumed to be finite and 
$P$ is stochastic, we have 
$\lambda_0(P)=1$, but  its restriction to a subset will no more be stochastic, and in this case
the spectral radius is $< 1$.
Then 
$$
G(x,y|z) = \sum_{n=0}^{\infty} p^{(n)}(x,y)\, \frac{1}{z^{n+1}} \quad \text{for }\; z \in \C\,,\;
|z| > \lambda_0(P).
$$
The following is well-known.

\begin{lem}\label{lem:poles} If $P$ is stochastic, $|X|=N$ and 
$$
1= \lambda_0 > \lambda_1 \ge \dots \ge \lambda_{N-2} \ge \lambda_{N-1} \ge -1
$$ 
are the eigenvalues of $P$ (which are all real), then
let $h_0\,,\dots, h_{N-1}$ be an orthonormal system of real-valued 
eigenfunctions (= right eigenvectors) of $P$ in $\ell^2(X,m_X)$.
That is,  $Ph_i = \lambda_i \, h_i$ and 
\begin{equation}\label{eq:delta}
\sum_{x \in X} h_i(x) h_j(x)\,m_X(x) = \delta_{i,j} \AND
\sum_{i=0}^{N-1} h_i(x)h_i(y)\, m_X(y) = \delta_{x,y}\,;
\end{equation}
in particular, $h_{0} \equiv 1\big/\sqrt{m_X(X)}$. 
Then for $x,y \in X$,
\begin{equation}\label{eq:Gxyz}
G(x,y|z) = \sum_{i=0}^{N-1} h_i(x)h_i(y)\, m_X(y) \, \frac{1}{z-\lambda_i}\,.
\end{equation}
Thus, all poles of $G(x,y|z)$ are simple 
and contained in $\spec(P)$. Furthermore,
$$
\sum_{x \in X} G(x,x|z) = \sum_{j=1}^{N} \frac{1}{z-\lambda_j}.
$$
\end{lem}

\begin{proof}
The fact that the eigenvalues are all real follows from reversibility, and that 
$1$ is a simple eigenvalue is a well-known consequence of the Perron-Frobenius theorem.

Let $H$ be the $(X \times \{0,\dots,N-1\})$-matrix whose $(x,i)$-entries are $h_i(x)$. 
Then $H^{-1}$ is the $(\{0,\dots,N-1\} \times X)$-matrix whose $(j,y)$ entries are
$h_j(y) m(y)$. If $D$ is the $(\{0,\dots,N-1\} \times \{0,\dots,N-1\})$-diagonal matrix
whose $(i,i)$-entries are $\lambda_i\,$, then 
$$
P = H\,D\,H^{-1}\,,\quad \text{whence} \quad  \Gc(z) = H\,(z\, I - D)^{-1}\,H^{-1}.
$$
This proves the formula for $G(x,y|z)$, and \eqref{eq:delta} yields the formula for
the trace of $\Gc(z)$.
\end{proof}

\begin{dfn}\label{def:locspec}
The \emph{local spectrum at} $x \in X$ is the set $\mathsf{S}_x=\mathsf{S}_x(P)$ 
of poles of $G(x,x|z)$. 
\end{dfn}

By Lemma \ref{lem:poles}, 
$$
\spec(P) = \bigcup_{x \in X} \mathsf{S}_x\,.
$$

The following is another consequence of Lemma \ref{lem:poles}.

\begin{lem}\label{lem:poles2}
We have $\lambda \in  \mathsf{S}_x$ if and only if there is an eigenfunction
$f$ with $Pf = \lambda\, f$ and $f(x) \ne 0$.
\end{lem}

\begin{proof}
The coefficient of $1/(\lambda - z)$ in the expression  \eqref{eq:Gxyz} for $G(x,x|z)$ is
$
\sum_{i : \lambda_i = \lambda} m(x)h_i(x)^2\,.
$
It is non-zero if and only of $h_i(x) \ne 0$ for at least one of those $i$. 
\end{proof}

Next, we define
$$
F(x,y|z) = \frac{G(x,y|z)}{G(y,y|z)}.
$$
Since $X$ is finite, these functions are rational. In particular, the last fraction will
allow some cancellations, and $z=1$ is not a pole of $F(x,y|z)$.
In probabilistic terms, let
$$
f^{(n)}(x,y) = \Prob_x[Z_n = y\,, Z_k \ne y \;\text{for}\; 0 \le k < n]
$$
be the probability that the first arrival of $(Z_n)$ in $y$ occurs at time $n$, when the starting 
point is $x$. In particular, $f^{(0)}(x,x)=1$ and $f^{(n)}(x,x)=0$ for $n \ge 1$. Then
\begin{equation}\label{eq:F1}
F(x,y|z) = \sum_{n=0}^{\infty} f^{(n)}(x,y)\, \frac{1}{z^n} \quad \text{for }\; z \in \C\,,\;
|z| > \lambda_0(P).
\end{equation}
Consider the restriction $P_{X\!-y}$ of $P$ to $X \setminus \{y\}$. Outside of 
$X \setminus \{y\}$,
we set its matrix entries to~$0$. This is also self-adjoint on $\ell^2(X,m_X)$.
As above, there is the associated resolvent $\Gc_{X\!-y}(z)$. We have 
$\lambda_0(P)=1 > \lambda_0(P_{X\!-y})$.
The power series in \eqref{eq:F1} 
converges for $|z| >  \lambda_0(P_{X\!-y})$,  since in this case
\begin{equation}\label{eq:F2}
F(x,y|z) = \sum_{w \in X \setminus \{y\}} G_{X\!-y}(x,w|z) p(w,y) \quad \text{when }\, x \ne y.
\end{equation}
By analytic continuation, this formula holds for all 
$z \in \C \setminus \spec(P_{X\!-y})$.\footnote{We shall often use equations derived via combinatorics 
of paths and their probabilistic counterpart as in \cite{WMarkov}, a priori valid for $|z| > 1$,
which then remain valid  by analytic continuation in $\C$ outside the respective set of 
poles.\label{fn:fn2}}
Similarly, we define $U(x,x|z)$ via
\begin{equation}\label{eq:U1}
zG(x,x|z) = \frac{1}{1 - U(x,x|z)}.
\end{equation}
If 
$$
u^{(n)}(x,x) = \Prob_x[Z_n = x\,, Z_k \ne x \;\text{for}\; 1 \le k < n]
$$
denotes the probability that the random walk starting at $x$ first returns to $x$
at time $n$, then for $|z| > \lambda_0(P_{X\!-x})$
$$
U(x,x|z) = \sum_{n=1}^{\infty} u^{(n)}(x,y)\, \frac{1}{z^n}\,,
$$
We then have (recalling that $p(x,x)=0$ by the assumed absence of loops)
\begin{equation}\label{eq:U2}
zU(x,x|z) = \sum_y p(x,y) F(y,x|z) = 
\sum_{y, w \in X \setminus \{x\}} p(x,y) 
G_{X\!-x}(y,w|z) p(w,x)
\end{equation}
by \eqref{eq:F2}. 
The analogous quantities and functions for $P_*$ will carry the index $*$, and those for $Q$ 
will carry the index~$V$. 
In a large body of work, the $n^{\textrm{th}}$ term of the above power series is paired with $z^n$,
see e.g. \cite{WMarkov}.
Here, it is more convenient to use the reciprocal. 
The function
\begin{equation}\label{eq:phi}
\varphi(z) = \frac{1}{F_{V}(a,b|z)},
\end{equation}
will be of crucial importance. For this rational function, we have in particular 
$\varphi(0)=0$ and $\varphi(\infty) = \infty$. (The rational functions appearing here
should be considered on $\widehat \C = \C \cup \{ \infty\}$.)

Let $Q_{V\!-b}$ be the
restriction of $Q$ to $V \setminus \{b\}$. It is strictly substochastic in at least one row,
and again reversible. We write $G_{V\!-b}(u,v|z)$ for the $(u,v)$-element of 
$\mathcal{G}_{V\!-b}(z) = (z\,I_{V\!-b} - Q_{V\!-b})^{-1}$, where $u, v \in V\!\!-\!b$, and we extend
this to $V$ by assigning value $0$ when $u=b$ or $v=b$. (We use the same notation with $a$ 
in the place of $b$.) Then by \eqref{eq:F2}
\begin{equation}\label{phiG}
1/\varphi(z) = \sum_{v \in V\!\!-\!b} G_{V\!-b}(a,v|z) q(v,b).
\end{equation}

Similarly to the above notation we have the restriction $Q_{V^o}$ of $Q$ 
to $V^o$ and the associated Green kernel $G_{V^o}(u,v|z)$.
Note that for $|z| > \lambda_0(Q_{V\!-b}) = \lambda_0(Q_{V\!-a})$ and $u, v \in V \setminus \{ b\}$
$$
G_{V\!-b}(u,v|z) =
\sum_{n=0}^{\infty} \Prob_u[Y_n = v\,,Y_k \ne b \;\text{for}\; k<n]\,\frac{1}{z^{n+1}}\,,
$$
and analogously for $|z| > \lambda_0(Q_{V^o})$ and $u, v \in V^o$
$$
G_{V^o}(u,v|z) =
\sum_{n=0}^{\infty} \Prob_u[Y_n = v\,,Y_k \ne a, b \;\text{for}\; k<n]\,\frac{1}{z^{n+1}}\,.
$$
We shall also need
\begin{equation}\label{eq:psi-theta}
\begin{aligned}
\psi(z) &= \sum_{u \in V \setminus \{ a\}} q(a,u) F_{V\!-a}(u,b|z)
= \sum_{u,v \in V^o} q(a,u) G_{V^o}(u,v|z) q(v,b) + q(a,b) \quad \text{and}\\
\theta(z) &= \sum_{u \in V \setminus \{ b\}} q(a,u) F_{V\!-b}(u,a|z) =
\sum\limits_{u, v \in V^o} q(a,u) G_{V^o}(u,v|z) q(v,a)\,,
\end{aligned}
\end{equation}
compare with formula \eqref{eq:F2}.

\begin{lem}\label{lem:Fab}
If $z \notin \spec(Q_{V^o})$ then 
$$
F_{V}(a,b|z) = G_{V\!-b}(a,a|z)\sum_{u \in V \setminus \{ a\}} q(a,u) F_{V\!-a}(u,b|z)\\
=\frac{\psi(z)}{z - \theta(z)}\,,
$$
which is finite except possibly when $\theta(z)=z$, 
in which case $z \in \spec(Q_{V\!-b})$.
\end{lem}

\begin{proof}
For $|z| > 1$, the first formula is obtained via ``combinatorics of paths''.
Any path in $V$ that starts at $a$ and visits $b$ only at the end can be decomposed
as follows: it first may return a number of times to $a$ without visiting $b$,
until the last visit in $a$. Second, it moves to some $u \in V \setminus \{a\}$, and 
continues from there without returning to $a$ until the first arrival in $b$. The 
first part leads to the term
\begin{equation}\label{eq:GVb}
G_{V\!-b}(a,a|z) = \frac{1}{z - zU_{V\!-b}(a,a|z)} \quad \text{with}\quad
zU_{V\!-b}(a,a|z) = \theta(z)\,,
\end{equation}
compare with \eqref{eq:U2}. For the second part, see \eqref{eq:psi-theta}, resp.
\eqref{eq:F2}. Again, the formulas extend to $\widehat \C \setminus \spec(Q_{V^o})$ by analytic
continuation.
\end{proof}

The following will be important for determining the spectral gap
in \S \ref{sec:spec}. Recall that we assume connectedness of $V \setminus \{ b\}$.

\begin{pro}\label{pro:phiprop}
 The rational function $\varphi(z)$ is finite, whence continuous, in the 
interval $(\lambda_0(Q_{V_o})\,,\,1]$, and strictly increasing at least
in the interval $[\lambda_0(Q_{V\!-b})\,,\,1]$. We have 
 $$
 \lambda_0(Q_{V_o}) < \lambda_0(Q_{V\!-b}) < 1\,, \quad 
 \varphi(1)=1 \AND \varphi\bigl(\lambda_0(Q_{V\!-b})\bigr) =  0.
 $$
If in addition $V^o$ is connected then $\varphi\bigl(\lambda_0(Q_{V^o})\bigr) =  -1$.
\end{pro}

\begin{proof} The matrices $Q$ and $Q_{V\!-b}$ are irreducible, so that the two
 strict inequalities follow from the Perron-Frobenius Theorem  (even beyond
reversibility), see {\sc Seneta}~\cite{Se}. 

It is clear that $\varphi(1)=1$.

Recall the formulas \eqref{eq:psi-theta} for $\psi(z)$ and $\theta(z)$. For 
$|z| > \lambda_0(Q_{V_o})$,
the functions $G_{V^o}(u,v|z)$ are convergent power series with non-negative coefficients in the 
variable $1/z$, so that $\psi(z)$ and $\theta(z)$ are positive and continuous in the
interval $(\lambda_0(Q_{V_o})\,,\,1]$. Furthermore, $z = \lambda_0(Q_{V\!-b})$ is a pole of 
$G_{V\!-b}(a,a|z)$, whence by \eqref{eq:GVb} it is a zero of $z - \theta(z)$. 
This yields $\varphi\bigl(\lambda_0(Q_{V\!-b})\bigr) =  0$. Since $z-\theta(z)$ is monotone increasing 
and $\psi(z)$ is decreasing for $z > \lambda_0(Q_{V_o})$, it follows that 
$\varphi(z)$ is increasing in $[\lambda_0(Q_{V\!-b})\,,\,1]$. 

Now assume that $V^o$ is connected. 
For this proof, let us write  $\lambda_0=\lambda_0(Q_{V^o})$. Lemma \ref{lem:poles} and \eqref{eq:Gxyz} 
apply here with the only exception that $\lambda_0 \ne 1$.
Again by the Perron-Frobenius Theorem,  $\lambda_0$ is a simple eigenvalue of $Q_{V^o}$, 
and the associated normalised eigenfunction $h_0$ is $> 0$ on all of $Q_{V^o}$. Furthermore,
the latter is invariant under $\gamma$, so that $Qh_0(a) = \sum_v q(a,v)h_0(v) = Qh_0(b)$.
The expansion formula \eqref{eq:Gxyz} implies
$$
\lim_{z \to \lambda_0} (z-\lambda_0) G_{V^o}(u,v|z) = h_0(u)h_0(v)m_V(v)\,.
$$
Therefore (with $u, v$ ranging in $V^o$), using reversibility 
$$
\begin{aligned}
\lim_{z \to \lambda_0} \varphi(z) &= 
\lim_{z \to \lambda_0} \frac{(z-\lambda_0)z - \sum_{u,v} q(a,u)\,(z-\lambda_0) G_{V^o}(u,v|z)\,q(v,a)}
{(z-\lambda_0)q(a,b) - \sum_{u,v} q(a,u)\,(z-\lambda_0) G_{V^o}(u,v|z)\,q(v,b)}\\[3pt]
&=-\frac{\sum_{u,v} q(a,u)h_0(u)\,h_0(v)m_V(v)q(v,a)}{\sum_{u,v} q(a,u)h_0(u)\,h_0(v)m_V(v)q(v,b)}\\[3pt]
&=-\frac{\sum_{u,v} q(a,u)h_0(u)\,h_0(v)m_V(a)q(a,v)}{\sum_{u,v} q(a,u)h_0(u)\,h_0(v)m_V(b)q(b,v)}\\[3pt]
&=-\frac{m_V(a)\bigl(Qh_0(a)\bigr)^2}{m_V(b)Qh_0(a)Qh_0(b)} = -1\,. 
\end{aligned}
$$
This concludes the proof.
\end{proof}

In certain cases, the last proposition may also be valid when $V^o$ is not connected. See 
e.g. Example \ref{exa:mple2} in \S \ref{sec:examples} below, where $V^o$ has two components
which are isomorphic.

The following formula is quite important for us.

\begin{pro}\label{pro:G} For $x,y \in X \subset X[V]$,
$$
G_*(x,y|z) = G\bigl(x,y|\varphi(z)\bigr)\big/\psi(z).
$$
\end{pro}
The proof is exactly as in \cite[(9.90)]{WMarkov}. There, only the special case is
considered where $V$ is a finite path of length $N$, but the adaptation is immediate. 
One also has to take into account the replacement of $z$ with $1/z$. As mentioned in 
footnote \ref{fn:fn2}, 
via this proof the formula first holds for $|z| > 1$, but it then extends to $\widehat \C$
by analytic continuation of the involved rational functions. 
The formula also holds when $X$ is infinite, with some care regarding the spectrum,
which then does no more consist of a finite collection of poles of the Green functions.

\begin{lem}\label{lem:bvalue}
Let $\alpha, \beta, z \in \C$ and consider the boundary value 
problem
$$
Qf = z\, f \; \text{ on }\; V^o\,, \quad f(a) = \alpha\,,\; f(b) = \beta.
$$
If $z \notin \spec(Q_{V^o})$ then 
it has the unique solution
$$
\begin{gathered}
f(u) = \alpha \, F_{V\!-b}(u,a|z) + \beta\, F_{V\!-a}(u,b|z)\,, \quad \text{where for }\; 
u \in V^o\\
F_{V\!-b}(u,a|z) =  \sum_{v \in V^o} G_{V^o}(u,v|z) q(v,a) \AND
F_{V\!-a}(u,b|z) = \sum_{v \in V^o} G_{V^o}(u,v|z) q(v,b). 
\end{gathered}
$$
\end{lem}

\begin{proof} With the values $F_{V\!-b}(a,a|z)=1$ and $F_{V\!-b}(b,a|z)=0$,
we have $QF_{V\!-b}(\cdot,a|z) = z\, F_{V\!-b}(\cdot,a|z)$ on $V^o$.
Exchanging the roles of $a$ and $b$, we have the analogous identity;
compare with \eqref{eq:F1} and \eqref{eq:F2}.
This shows that the proposed function $f$ is indeed a solution with the required
boundary values. 

To show uniqueness, we need to show that necessarily $f \equiv 0$ when $\alpha=\beta=0$.
In this case, the restriction of $f$  to $V^o$ 
must satisfy  $(z\,I_{V^o} -Q_{V^o})f_{V^o} = 0$, and the additional hypothesis 
$z \notin \spec(Q_{V^o})$ yields uniqueness.
\end{proof}

We now consider the specific case when $X = S = S_N$ is the star 
with central vertex $x$, leaves $y_1\,,\dots, y_N$ and edges 
$e_j = [x,y_j]$, $j=1\,\dots, N$. We set $p_j = p(x,y_j) = a_S(e_j)$. 
The reversed probabilites are $p(y_j\,,x)=1$, but for the moment they will play 
only a secondary role. We consider
$P_*$ on $S[V]$. We may assume without loss of 
generality that along each edge $e_j$ of $S$ ($j \ge 1$), we have $e_j^a=x$ and $e_j^b=y_j\,$.

\smallskip

Equation \eqref{eq:Upsi} and statment (b) of the following proposition 
(leading to the subsequent theorem) first appeared in \cite{St}; we provide a proof within our 
terminology.
 
\begin{pro}\label{pro:star}
\emph{(a)} Let $\lambda^* \in \C \setminus \spec(Q_{V^o})$. Suppose that $f^*: S[V] \to \C$ is a 
non-zero function such that
\begin{equation}\label{eq:lambda*f*}
P_*f^* = \lambda^* \, f^* \quad \text{on }\; S[V] \setminus \{ y_1\,,\dots,y_N\}.
\end{equation}
Then the restriction $f = f^*|_S$ satisfies
\begin{equation}\label{eq:Upsi}
\bigl(\lambda^*- \theta(\lambda^*)\bigr)f(x) 
= \psi(\lambda^*)\, Pf(x).
\end{equation}
Thus, if $\psi(\lambda^*) \ne 0$ then
$$
Pf(x) = \varphi(\lambda^*)\,f(x).
$$
\emph{(b)} Conversely, suppose that $f: S \to \C$ is a non-zero function such that 
$Pf(x) = \lambda \, f(x)$, and that there is $\lambda^* \in \C \setminus \spec(Q_{V^o})$
such that $\varphi(\lambda^*) = \lambda$. Then $f$ extends to a unique function $f^*$
on $S[V]$ such that \eqref{eq:lambda*f*} holds.
\\[4pt]
\emph{(c)} If $\lambda^*\in 
\C \setminus \spec(Q_{V^o})$ satisfies
$$
\psi(\lambda^*) = 0 \AND \theta(\lambda^*) = \lambda^* 
$$
then \emph{every} function $f: S \to \R$ extends uniquely to a function 
$f^*$ such that \eqref{eq:lambda*f*} holds.
\end{pro}

\begin{proof} (a) For any $j \in \{1,\dots, N\}$, we write  $u_j$ and $v_j$
for the copy of $u, v \in V^o$ in the substitution of $V$ along the edge $[x,y_j]$. The 
restriction of $f^*$ to $V_j$ solves the boundary problem of Lemma \ref{lem:bvalue}
corresponding to $f(a)=f(x)$ and $f(b) = f(y_j)$. By uniqueness of the solution,
\begin{equation}\label{eq:solution}
f^*(u_j) = f(x) \, F_{V\!-b}(u,a|\lambda^*) + f(y_j)\, F_{V\!-a}(u,b|\lambda^*).
\end{equation}
Therefore, 
$$
\lambda^* f(x) = \sum_{j=1}^N p_j 
\sum_{u \in V} q(a,u)\Bigl(F_{V\!-b}(u,a|\lambda^*)f(x) + F_{V\!-a}(u,b|\lambda^*)f(y_j)
\Bigr)\,.
$$
Reordering the terms and using the last formula of Lemma \ref{lem:bvalue}, we get
\begin{equation}\label{eq:reordered}
\begin{aligned}
&\biggl(\lambda^* - \sum_{j=1}^N p_j 
\overbrace{\sum_{u, v \in V^o} q(a,u) G_{V^o}(u,v|\lambda^*) q(v,a)}^{\displaystyle = \theta(\lambda^*)}
\biggr) f(x) \\ 
&\qquad\qquad\qquad\qquad= \sum_{j=1}^N p_j 
\underbrace{\biggl(\sum_{u,v \in V^o} q(a,u) G_{V^o}(u,v|\lambda^*) q(v,b) + q(a,b)\biggr)}_{\displaystyle 
= \psi(\lambda^*)} f(y_j)\,. 
\end{aligned}
\end{equation}
(b) Given $f$, in order to have \eqref{eq:lambda*f*} we must define $f^*$ by 
\eqref{eq:solution}. Then $P_*f^* = \lambda^* \, f^*$ on 
$S[V] \setminus \{x,  y_1\,,\dots,y_N\}$ by Lemma \ref{lem:bvalue}.
From $\varphi(\lambda^*) = \lambda$ 
we get $\lambda^* - \theta(\lambda^*) = \lambda \psi(\lambda^*)$, whence (using the same formulas
as above)
$$
\begin{aligned} 
P_* f^*(x) &= \sum_{j=1}^N p_j 
\sum_{u \in V} q(a,u)\Bigl(F_{V\!-b}(u,a|\lambda^*)f(x) + F_{V\!-a}(u,b|\lambda^*)f(y_j)
\Bigr)\\
&= \theta(\lambda^*) f(x) + \psi(\lambda^*) Pf(x) = \lambda^* f^*(x). 
\end{aligned}
$$
(c) follows from the last  two lines.
\end{proof}

\begin{thm}\label{theorem:eigen}
\emph{(i)} Let $\lambda^* \in \spec(P_*)$ with associated non-zero eigenfunction $f^*$ on $X[V]$, 
that is $P_*f^* = \lambda^*\,f^*$.
Suppose that $\lambda^* \notin \spec(Q_{V^o})$. Then the restriction $f = f^*|_X$
satisfies \eqref{eq:Upsi} at every $x \in X \subset X[V]$.
If 
$\psi(\lambda^*) \ne 0$ then
$$
Pf = \lambda\, f\,, \quad \text{where}\quad \lambda = \varphi(\lambda^*)\,,\AND f \not\equiv 0 \; 
\text{ on }\; X.
$$
If $\psi(\lambda^*) = 0$ then also $\theta(\lambda^*) = \lambda^*$, whence 
$\lambda^* \in \spec(Q_{V\!-b})$.
\\[4pt]\emph{(ii)}
Conversely, suppose that $f$ is a non-zero function on $X$ such that $Pf = \lambda f$ and 
that $\lambda^*\in  [-1\,,\,1] \setminus \spec(Q_{V^o})$ is a solution of 
$\varphi(\lambda^*) = \lambda$.
Then $f$ extends to a unique function $f^*$ on $X[V]$ such that $P_*f^* = \lambda^*\, f^*$.

Furthermore, if $\lambda^*\in  [-1\,,\,1] \setminus \spec(Q_{V^o})$ satisfies
$$
\psi(\lambda^*) = 0 \AND \theta(\lambda^*) = \lambda^* 
$$
then  \emph{every} function $f: X \to \R$ extends to a unique function 
$f^*$ such that $P_*f^* = \lambda^*\, f^*$, and $\lambda^* \in \spec(Q) \cap \spec(Q_{V\!-b})$.
\end{thm}

\begin{proof}
(i) Recall that the factors on both sides of \eqref{eq:Upsi} are finite when
$\lambda^* \notin \spec(Q_{V^o})$, see \eqref{eq:F2} and \eqref{eq:GVb}.
We can apply Proposition \ref{pro:star} to the neighbourhood star $S(x)$ at every
vertex of $X \subset X[V]$. 
Suppose that $f \equiv 0$ on $X$. Then on the substituted version of $V$ along each edge of $X$,
the function $f^*$ satisfies the boundary problem of Lemma \ref{lem:bvalue} with $z = \lambda^*$
and boundary values $0$. Therefore $f^* \equiv 0$, a contradiction.
Finally, if $\psi(\lambda^*)=0$ and $\theta(\lambda^*) \ne \lambda^*$ then $f \equiv 0$. 
We just saw that
this contradicts the hypothesis that $\lambda^* \notin \spec(Q_{V^o})$.
This proves (i).
\\[3pt] 
(ii) follows directly from Proposition \ref{pro:star}, except for the fact that 
$\psi(\lambda^*) = \lambda^* - \theta(\lambda^*) = 0$ implies 
$\lambda^* \in \spec(Q)$. To see this, let $X = \{x,y\}$ be the graph with just one 
edge $[x,y]$ and $p(x,y)=p(y,x)=1$. Then $X[V]=V$ and $P_*=Q$, and by the above,
$\lambda^* \in \spec(P_*)$.   
\end{proof}

\begin{rmk}\label{rmk:Str-gap} 
Our equation \eqref{eq:Upsi} is the same as the fundamental equation (3.5) of 
\cite{St}, with additional details concerning its ingredients in terms of walk 
generating functions. The central Theorem 3.1 of \cite{St} has a gap. Namely, apart from the missing 
minus signs of the Laplacians in  \cite[(3.5)]{St}, what is overlooked is that 
in the notation \& terminology of  that paper, it may occur that $A_1(\lambda')=0$, which adds 
a third alternative; compare with the last 3 lines of the preceding page 2052.
In our notation, this corresponds to the case when $\psi(\lambda^*) = 0$ and 
$\theta(\lambda^*) = \lambda^*$, for which we have Example \ref{exa:mple3} in~\S \ref{sec:examples}.
\end{rmk}

Recall that $\varphi$ is a rational function, and it poles are zeroes of $F_{V}(a,b|z)$. 
The set
$$
\varphi\bigl(\spec(Q_{V^o})\bigr) 
= \{ \varphi(\tilde\lambda) : \tilde\lambda \in \spec(Q_{V^o}) \} \subset 
\R \cup \{\infty \}
$$
is finite. If $\lambda \in \spec(P)$ is not contained in this set, then every solution
$\lambda^*$ of $\varphi(\lambda^*)=\lambda$ yields an eigenfunction and hence is contained in 
$\spec(P_*)$. In particular, all those solutions must be real and contained in $[-1\,,\,1]$.
Now every $\lambda \in [-1\,,\,1]$ can appear as an eigenvalue of some $P$ on some bipartite
graph $X$, see Lemma \ref{lem:lambda}. 

\begin{cor}\label{cor:allreal}
If $\lambda \in [-1\,,\,1]$ then every solution
$\lambda^*$ of $\varphi(\lambda^*) = \lambda$ must be real and contained in $[-1\,,\,1]$.
\end{cor}

\begin{proof}
A priori, this must hold for all $\lambda \in [-1\,,\,1] \setminus \varphi\bigl(\spec(Q_{V^o})\bigr)$.
Since the exceptional set is finite, the property extends by continuity to 
$\lambda \in [-1\,,\,1] \cap \varphi\bigl(\spec(Q_{V^o})\bigr)$.
\end{proof}

Another consequence of Lemma \ref{lem:lambda} is the following.

\begin{lem}\label{lem:psine0}
Let $\lambda^* \in [-1\,,\,1]$. If 
$\varphi(\lambda^*)= \lambda \in  [-1\,,\,1]$ then $\psi(\lambda^*) \ne 0$. 

Furthermore, if $\lambda^*$ is not a pole of $\psi(z)$, then it is a simple solution of
the equation \hbox{$\varphi(\lambda^*)= \lambda$.}
\end{lem}

\begin{proof}  There are a graph $X$ (a circle of even length) 
and an associated reversible
transition matrix $P$ such that $\lambda \in \spec(P)$. Therefore $\lambda \in \mathsf{S}_x$
for some $x \in X$. That is, in the expansion \eqref{eq:Gxyz} of $G(x,x|z)$, the term
$\frac{c}{z-\lambda}$ appears with a coefficient $c > 0$. Then by Proposition \ref{pro:G}
the term 
$$
\frac{c}{\varphi(z)-\lambda} \cdot\frac{1}{\psi(z)}
$$
appears in the induced expansion of $G_*(x,x|z)$. If $\varphi(\lambda^*)=\lambda$ and 
$\psi(\lambda^*) = 0$ then $\lambda^*$ would be a pole with multiplicity $> 1$ of $G_*(x,x|z)$,
which is excluded. The same argument shows that $\lambda^*$ is a simple solution when $\lambda^*$ is not 
a zero of $1/\psi(z)$.
\end{proof}

Thus, the most important link between the spectra of $P$ and $P_*$ is the equation 
$\lambda = \varphi(\lambda^*)$. The function $\varphi(\lambda^*)$ is rational,
and for given $\lambda$ there will be several solutions for $\lambda^*$,
so that a given eigenfunction $f$ of $P$ can have different extensions corresponding to
different solutions~$\lambda^*$. We are going to see 
further parts of the spectrum of $P_*$ besides the solutions of $\varphi(\lambda^*) = \lambda$
with further sources of multiplicities.

\section{Embedding $\spec(Q)$ and $\spec(Q_{V^o})$ into $\spec(P_*)$}\label{sec:into}

There are simple and natural candidates for 
eigenvalues and non-zero eigenfunctions of $P_*$ coming directly from $Q$, based on the next proposition.
Recall the automorphism $\gamma$ which 
exchanges $a$ and $b\,$: for any $f:V \to \C$ we define $f^{\gamma}$ by
$f^{\gamma}(v) = f(\gamma v)$. Let $r$ be the order of the automorphism $\gamma$, an even number. 

\begin{pro}\label{pro:specQ} Let $\lambda \in \spec(Q)$ be an eigenvalue with multiplicity 
$\nu = \nu_Q(\lambda) \ge 1$. 
Then the associated eigenspace $\mathcal{N}(\lambda)$ has a basis $\{ f_1\,,\dots, f_{\nu} \}$
of eigenfunctions with one of the following properties.
\begin{itemize}
 \item[\rm (I)] $f_j(a)=f_j(b)= 0$ for all $j \in \{ 1,\dots, \nu\}$, or\\[-10pt]
 \item[\rm (II)] $f_j(a)=f_j(b)= 0$ for all $j \in \{ 1,\dots, \nu-1\}$, while 
 $f_{\nu}(a)=f_{\nu}(b) = 1$ and $f_{\nu}^{\gamma}= f_{\nu}\,$,
 or\\[-10pt]
 \item[\rm (III)] $f_j(a)=f_j(b)= 0$ for all $j \in \{ 1,\dots, \nu-1\}$, while 
 $f_{\nu}(a)=-f_{\nu}(b) = 1$ and $f_{\nu}^{\gamma}= -f_{\nu}\,$,
 or\\[-10pt]
 \item[\rm (IV)] $f_j(a)=f_j(b)= 0$ for all 
 $j \in \{ 1,\dots, \nu-2\}$, while $\; f_{\nu}(a) = f_{\nu-1}(b) = 1$, 
 $f_{\nu}(b) = f_{\nu-1}(a) =0\,$,
 $\;f_{\nu}^{\gamma}= f_{\nu-1}\,$, and $\;f_{\nu-1}^{\gamma}= f_{\nu}\,$.
\end{itemize}
Set $\nu_Q'(\lambda)= \nu$ in case \emph{(I)}, $\nu_Q'(\lambda) =\nu-1$ in cases \emph{(II)} and 
\emph{(III)}, and $\nu_Q'(\lambda) = \nu-2$ in case \emph{(IV)}.
When $\nu_Q'(\lambda) \ge 1$, we also have $\lambda \in \spec(Q_{V^o})$. In case \emph{(IV)}, 
$\lambda \in \spec(Q_{V-b})$.
\end{pro}

\begin{proof} We start with an arbitrary base $\{ \bar f_1\,,\dots, \bar f_{\nu} \}$ of 
$\mathcal{N}(\lambda)$.
If it has property (I) then we set $f_j = \bar f_j$ and are done. 

Otherwise there is $j$ such that at least one of 
$\bar f_j(a)$ and $\bar f_j(b)$ is $\ne 0$. 
\\[4pt]
\emph{Case 1.} There is $j$ such that $\bar f_j(a) \ne \pm \bar f_j(b)$.  By permuting the basis,
we can assume that $j=\nu$.
Then the functions $\bar f_{\nu}$ and $\bar f_{\nu}^{\gamma}$ are linearly independent 
(so that $\nu \ge 2$), and we can replace some  $\bar f_j$, $j < \nu$, by $\bar f_{\nu}^{\gamma}$. 
Then we reorder the new basis by exchanging the new $\bar f_j$ ($=\bar f_{\nu}^{\gamma}$) with 
$\bar f_{\nu-1}$.
The new basis (maintaining the notation $\bar f_j$ for the new elements after replacement 
and reordering) is $\{ \bar f_1\,,\dots, \bar f_{\nu-2}\,, \bar f_{\nu}^{\gamma}\,, \bar f_{\nu}\}$.  
Now we replace  $\{ \bar f_{\nu}^{\gamma}\,, \bar f_{\nu} \}$ by
$$
\biggl\{ \tilde f_{\nu-1} 
= \frac{1}{\bar f_{\nu}(a)-\bar f_{\nu}(b)}(\bar f_{\nu} - \bar f_{\nu}^{\gamma})\,,\;  
\tilde f_{\nu} = \frac{1}{\bar f_{\nu}(a)+\bar f_{\nu}(b)}(\bar f_{\nu} + \bar f_{\nu}^{\gamma}) \biggr\}
$$
so that $\tilde f_{\nu-1}(a)=-\tilde f_{\nu-1}(b)= \tilde f_{\nu}(a)=\tilde f_{\nu}(b)=1$. 
At this point we can define the  basis elements
\begin{equation}\label{eq:ftilde2}
\overset{\scriptscriptstyle\approx}{\!f}_{\nu} = \frac{1}{r}\sum_{k=0}^{r-1} \tilde f_{\nu}^{\gamma^k}
\AND 
\overset{\scriptscriptstyle\approx}{\!f}_{\!\nu-1} 
= \frac{1}{r}\sum_{k=0}^{r-1} (-1)^k \tilde f_{\nu-1}^{\gamma^k} 
\end{equation}
Then $f_{\nu}= \frac12 (\overset{\scriptscriptstyle\approx}{\!f}_{\nu} +
\overset{\scriptscriptstyle\approx}{\!f}_{\nu-1})$ and 
$f_{\nu}= \frac12 (\overset{\scriptscriptstyle\approx}{\!f}_{\nu} -
\overset{\scriptscriptstyle\approx}{\!f}_{\nu-1})$
have the stated properties.

The real vectors $\bigl( f_{\nu-1}(a), f_{\nu-1}(b) \bigr)$ and $\bigl( f_{\nu}(a),f_{\nu}(b) \bigr)$
are linearly independent. Therefore, for each $j \in \{ 1, \dots, \nu-2\}$
there are $\alpha_j\,, \beta_j \in \R$ such that 
$$
f_j  = \bar f_j - \alpha_j \, f_{\nu-1} - \beta_j\, f_{\nu}  
$$
satisfies $f_j(a)=f_j(b)=0$. In this way, we have found a new basis of 
$\mathcal{N}(\lambda)$ which has property (IV).
\\[4pt]
If we do not have (I) nor Case 1, so that $f_j(a) = \pm f_j(b)$ for all $j$, 
then we distinguish the following remaining cases.
\\[4pt]
\emph{Case 2.} $\bar f_j(a) = \bar f_j(b)$ for all $j$, and $\bar f_k(a) \ne 0$ for at least
one $k$. Then we exchange $\bar f_k$ with $\bar f_{\nu}$ and subsequently replace the new 
$\bar f_{\nu}$ by $\tilde f_{\nu}= \frac{1}{\bar f_{\nu}(a)}\bar f_{\nu}$. 
We define the final basis element 
$$
f_{\nu} = \frac{1}{r}\sum_{k=0}^{r-1} \tilde f_{\nu}^{\gamma^k}\,.
$$
Finally, for $j=1,\dots, \nu-1$, we replace 
$\bar f_j$ by $f_j = \bar f_j - \bar f_j(a)f_{\nu}$. 
Then $\{ f_1\,,\dots, f_{\nu}\}$ is a new basis of 
$\mathcal{N}(\lambda)$ which has property (II).
\\[4pt]
\emph{Case 3.} $f_j(a) = -f_j(b)$ for all $j$, and $f_k(a) \ne 0$ for at least
one $k$. Then we exchange $\bar f_k$ with $\bar f_{\nu}$ and subsequently replace the new 
$\bar f_{\nu}$ by $\tilde f_{\nu}= \frac{1}{\bar f_{\nu}(a)}\bar f_{\nu}$. 
Here, we define the final basis element 
$$
f_{\nu} = \frac{1}{r}\sum_{k=0}^{r-1} (-1)^k \tilde f_{\nu}^{\gamma^k}\,.
$$
Then we conclude as in Case 2 and get  a new basis of 
$\mathcal{N}(\lambda)$ which has property (III).
\\[4pt]
\emph{Case 4.} There are $j_1$ and $j_2$ such that $\bar f_{j_1}(a) = \bar f_{j_1}(b)\ne 0$
and $\bar f_{j_2}(a) = -\bar f_{j_2}(b)\ne 0$. Then we exchange 
$\bar f_{\nu}$ with $\bar f_{j_1}$ and $\bar f_{\nu-1}$
with $\bar f_{j_2}$. Subsequently, given this new numbering, we normalize: 
$\tilde f_{\nu-1}= \frac{1}{\bar f_{\nu-1}(a)}\bar f_{\nu-1}$ 
and $\tilde f_{\nu}= \frac{1}{\bar f_{\nu}(a)}\bar f_{\nu}\,$. 
Then we can proceed as in Case 1
to get a new basis which has property (IV).

\smallskip

Regarding the additional statements, if $\nu' \ge 1$ then the restriction of each of 
the first $\nu'$ eigenfunctions to $V^o$ is an eigenfunction of $Q_{V^o}$ with eigenvalue $\lambda$.

In case (IV), the restriction of $f_{\nu}$ to $V \setminus \{b\}$ is an eigenfunction of $Q_{V-b}\,,$
and the restriction of $f_{\nu-1}$ to $V \setminus \{a\}$ is an eigenfunction of $Q_{V-a}\,$.
\end{proof}

We now explain how in each of the four types of Proposition \ref{pro:specQ}, 
one can use the respective basis of $\mathcal{N}(\lambda)$ to construct linearly independent
eigenfunctions of $P_*$ with the same eigenvalue. In what follows, we always 
assume that $\lambda \in \spec(Q)$, $\nu = \nu_Q(\lambda)$ and $\nu' = \nu_Q'(\lambda)$.
For the following, compare with \cite[bottom of p. 2053]{St}.
\begin{lem}\label{lem:i}
In each of the four cases of Proposition \ref{pro:specQ}, if $\nu' > 0$ then each of the 
functions $f_j\,$, 
$j=1,\dots, \nu'$, can be extended to $|E_X|$ linearly independent eigenfunctions of $P_*$ 
with the same
eigenvalue.
\end{lem}

\begin{proof}
Let $e \in E_X$ and define $f_{j,e}^*$ as follows:
$$
f_{j,e}^*(e,v)=f_j(v) \; \text{ for }\; v \in V^o \AND f_{j,e}^* \equiv 0 \;\text{ on }\;
X[V] \setminus \{e\}\times V^o.
$$
The resulting $\nu' \, |E_X|$ functions are linearly independent eigenfunctions with 
eigenvalue~$\lambda$.
\end{proof}

\begin{lem}\label{lem:ii}
In case \emph{(II)} of Proposition \ref{pro:specQ}, the function $f_{\nu}$ can be extended to an
eigenfunction $f_{\nu}^*$ of $P_*$ with the same eigenvalue.
\end{lem}

\begin{proof}
We define
$$
f_{\nu}^* \equiv 1 
\;\text{ on } X \subset X[V]
\AND f_{\nu}^*(e,v) = f_{\nu}(v) \; \text{ for each }\;(e,v) \in E_X \times V^o.
$$
This is an eigenfunction of $P_*$ with eigenvalue $\lambda$. 
\end{proof}

\begin{lem}\label{lem:iii}
If $X$ is bipartite then in case \emph{(III)} of Proposition \ref{pro:specQ}, 
the function $f_{\nu}$ can be extended to an
eigenfunction $f_{\nu}$ of $P_*$ with the same eigenvalue.
\end{lem}

\begin{proof}
Let $X = X_1 \cup X_2$ be the bipartition of the vertex set of $X$. In this case, it is reasonable
to orient the edges such that $e^a \in X_1$ and $e^b \in X_2$ for any edge $e \in E_X$. 
We define
$$
\begin{gathered}
f_{\nu}^* \equiv 
1 \;\text{ on } X_1 \subset X[V]\,,\quad 
f_{\nu}^* \equiv 
-1\;\text{ on } X_2 \subset X[V]
\quad \text{and}\\
f_{\nu}^*(e,v) = f_{\nu}(v) \; \text{ for each }\;(e,v) \in E_X \times V^o.
\end{gathered}
$$
Once more, this is an eigenfunction of $P_*$ with eigenvalue $\lambda$.
\end{proof}

The last construction fails when $X$ is not bipartite.
The fourth case in the ``richest'' one. 

\begin{lem}\label{lem:iv} If $\lambda$ is of type \emph{(IV),} then for any 
 $x \in X$, the functions $f_{\nu}$ and $f_{\nu-1}$ can be used to
 construct an eigenfunction $f_x^*$ of $P_*$ with the same eigenvalue,
 such that $f_x^*(x) = 1$ and $f_x^*(e,\cdot) \equiv 0$ for all 
 $e \in E_X$ not incident with $x$.
 
 This provides $|X|$ linearly independent eigenfunctions.
\end{lem}

\begin{proof}
If $e \in E_X$ is an edge incident with $x$ then we set for every $v \in V^o$
$$
f_x^*(e,v)= \begin{cases}  f_{\nu}(v)\,,&\text{if }\; x=e^a\\
                           f_{\nu-1}(v)\,,&\text{if }\; x=e^b
            \end{cases}
$$
For all other edges $\tilde e \in E_X$, we  set
$f_x(\tilde e,v) = 0$ for $v \in V^o$. On $X \subset X[V]$, we set $f_x(y) =\delta_x(y)$. 
\end{proof}

In the last Lemma, 
\emph{every} function $f: X \to \R$ extends to an eigenfunction of $P_*$ with eigenvalue $\lambda$. 
This should be compared with the last statement of Theorem \ref{theorem:eigen}.

\smallskip

We see that when $X$ is bipartite, we always have $\spec(Q) \subset \spec(P_*)$ with possibly
increased multiplicities of the eigenvalues. When $X$ is not bipartite, and $\lambda \in \spec(Q)$
has property (III), then not necessarily $\lambda \in \spec(P_*)$ unless $\nu_Q'(\lambda) > 0$.

Now -- and as we shall see, more importantly -- we also examine whether 
the Dirichlet spectrum 
of $Q$, that is, $\spec(Q_{V^o})$ or parts thereof can be embedded into $\spec(P_*)$
via suitable extensions of eigenfunctions.

\begin{pro}\label{pro:specQVo} Let $\lambda \in \spec(Q_{V^o})$ be an eigenvalue with multiplicity 
$\nu_o = \nu_o(\lambda) \ge 1$. For any associated eigenfunction $f$ on $V^o$, 
we consider the normal derivatives\footnote{See e.g. 
\cite{St}, \cite[\S 2]{SY} for this terminology.}
$Qf(a) = \sum_v q(a,v) f(v)$ and analogously $Qf(b)$.

Then the associated eigenspace $\mathcal{N}_{V^o}(\lambda)$ has a basis $\{ f_1\,,\dots, f_{\nu_o} \}$
of eigenfunctions with one of the following properties.
\begin{itemize}
 \item[\rm (I$^o$)] $Qf_j(a)=Qf_j(b)= 0$ for all $j \in \{ 1,\dots, \nu_o\}$, or\\[-10pt]
 \item[\rm (II$^o$)] $Qf_j(a)=Qf_j(b)= 0$ for all $j \in \{ 1,\dots, \nu_o-1\}$, while 
 $Qf_{\nu_o}(a)=Qf_{\nu_o}(b) = 1$ and $f_{\nu_o}^{\gamma}= f_{\nu_o}\,$,
 or\\[-10pt]
 \item[\rm (III$^o$)] $Qf_j(a)=Qf_j(b)= 0$ for all $j \in \{ 1,\dots, \nu_o-1\}$, while 
 $Qf_{\nu_o}(a)=-Qf_{\nu_o}(b) = 1$ and $f_{\nu_o}^{\gamma}= -f_{\nu_o}\,$,
 or\\[-10pt]
 \item[\rm (IV$^o$)] $Qf_j(a)=Qf_j(b)= 0$ for all 
 $j \in \{ 1,\dots, \nu_o-2\}$, while 
 $\;Qf_{\nu_o}(a) = Qf_{\nu_o-1}(b) = 1$, $Qf_{\nu_o}(b) = Qf_{\nu_o-1}(a) =0\,$,
 $\;f_{\nu_o}^{\gamma}= f_{\nu_o-1}\,$,
 and $f_{\nu_o-1}^{\gamma} = f_{\nu_o}\,$.
 \end{itemize}
Set $\nu_o'(\lambda)= \nu_o$ in case \emph{(I$^o$)}, $\nu_o'(\lambda) =\nu_o-1$ 
in cases \emph{(II$^o$)} and 
\emph{(III$^o$)}, and $\nu_o'(\lambda) = \nu_o-2$ in case~\emph{(IV$^o$)}. 
When $\nu_o'(\lambda) \ge 1$ we also have $\lambda \in \spec(Q_V)$.
\end{pro}

The proof is practically the same as the one of Proposition \ref{pro:specQ}. 

\begin{rmk}\label{rmk:zero}
When $\nu_o'(\lambda) \ge 1$ in Proposition \ref{pro:specQVo}, resp. $\nu_Q'(\lambda) \ge 1$ 
in Proposition \ref{pro:specQ},   we have 
$\lambda \in \spec(Q_{V}) \cap \spec(Q_{V^o})$ and $\nu_o'(\lambda) = \nu_Q'(\lambda)$.

Indeed, when there is an eigenfunction $f$ of $Q_{V^o}$ with $Qf(a)=Qf(b)=0$
then we can extend it to $V$ with value $0$ at $a$ and $b$. This is then an eigenfunction of $Q$
with the same eigenvalue.
In this case, the extension of $f$ to $|E_X|$ linearly independent eigenfunctions of 
$P_*$ is already covered by Lemma \ref{lem:i}. \hfill $\square$
\end{rmk}

\begin{dfn}\label{def:balance} Let $f^*: X[V] \to \R$. For $e \in E_X$ and 
$v \in V^o$ we let $f^e(v) = f^*(e,v)$. The \emph{balance} of $f^*$ at $x \in X \subset X[V]$ 
is 
$$
\bal(f^*,x) = \sum_{e \in E_X: e^a = x} a_X(e)\, Qf^e(a) 
+ \sum_{e \in E_X: e^b = x} a_X(e)\, Qf^e(b)\,,
$$
and we say that $f^*$ is \emph{balanced} if $\;\bal(f^*,x) = 0\;$ for every $x \in X$.
\end{dfn}

The motivation for this definition is the following: if $f^*: X[V] \to \R$ is such that 
$Q_{V^o}f^e = \lambda\, f^e$ for each $e \in E_X$ then it  an
eigenfunction of $P_*$ if and only if $f^*|_X \equiv 0$ and $f^*$ is balanced.

Recall that a \emph{fundamental cycle base} of the graph $X$ is obtained by starting 
with a spanning tree $T$ of $X$. Then  each edge of $X$ outside the tree gives rise to a 
cycle when added to the tree, and the resulting cycles are a fundametal base.

For the following two lemmas (whose core is contained in \cite{St})
concerning the extension of eigenfunctions 
of $Q_{V^o}$ with $Qf(a) = \pm Qf(b) = 1$, we invert the order between (II$^o$) and (III$^o$),
starting with the easier situation.

\begin{lem}\label{lem:odd}
Let $\lambda \in \spec(Q_{V^o})$ and according to 
Proposition \ref{pro:specQVo}, let $f$ be an associated eigenfunction of $Q_{V^o}$
with $Qf(a) = -Qf(b) = 1$, so that $f = f_{\nu_o}$ of case \emph{(III$^o$).} 

Then for every cycle $C = [x_0\,, x_1\,, \dots, x_k=x_0]$ ($k \ge 3$) 
in the host graph $X$, the function $f$ can be used to construct an eigenfunction $f_C^*$ of $P_*$ 
with the same eigenvalue which is $=0$ on $X[V]$ except for the vertices coming from 
$V^o$ substituted along the edges of the cycle. 

In particular, if $C_1\,,\dots, C_{\ell}$ is a fundamental cycle base of the graph $(X,E_X)$ then this
yields $\ell = |E_X| - |X| +1$ linearly independent eigenfunctions of $P_*$ with
eigenvalue $\lambda$. 
\end{lem}

\begin{proof}
Let $e_j= [x_{j}\,,x_{j+1}]$ (with $j+1\ \modulo k$) for $j=0\,,\dots, k-1$.
According to whether the orientation of $e_j$ is in the direction of the 
cycle or not, we define 
$$
\sgn_C(e_j) = \begin{cases} 1\,,&\text{if }\; e^a= x_j\,,\\
                        -1\,,&\text{if }\; e^a= x_{j+1}\,.
              \end{cases}
$$
Then we define for $v \in V^o$
$$
f_C^*(e_j,v) = \frac{\sgn_C(e_j)}{a_X(e_j)} f(v)\,,\;j= 0, \dots, k-1,
$$
where (recall) $a_X(e)$ is the conductance of $e \in E_X\,$.
On all other vertices of $X[V]$ we set $f_C^* =0$.
The balance of $f_C^*$ only needs to be checked at the vertices of the cycle.
For $j=0, \dots, k-1$, with indices always taken $\modulo k$, we have 4 possible cases at $x_j\,$:
$$
\begin{aligned} 
 \bal(f^*,x_j) &= \sgn_C(e_j) \cdot 
 \begin{cases} Qf(a)\,, &\sgn_C(e_j)=1\\Qf(b)\,, &\sgn_C(e_j)=-1 \end{cases}\Biggr\}\\  
               &\quad + \sgn_C(e_{j-1}) \cdot 
 \begin{cases} Qf(b)\,, &\sgn_C(e_{j-1})=1\\Qf(a)\,, &\sgn_C(e_{j-1})=-1 \end{cases}\Biggr\}\\                
               &=  (\pm 1) -  (\pm 1) = 0\,.
\end{aligned}
$$

\smallskip

Now take the fundamental cycle base induced by a spaning tree $T$. 
If $g = \sum_i \alpha_i\, f_{C_i}^* \equiv 0$ then it is $0$ along the 
``lonely edge'' of each $C_i$ (the one added to $T$), whence $\alpha_i=0\,$.
\end{proof}

We remark that Lemma \ref{lem:odd} also applies to the function 
$\overset{\scriptscriptstyle\approx}{\!f}_{\nu-1}$ of \eqref{eq:ftilde2} in case 
(IV$^o$), for which we shall use a different approach.

\begin{dfn}\label{def:nonbacktr}
A \emph{nonbacktracking closed path} in $X$ is a sequence  
$C \!=\! [x_0\,, x_1\,,..., x_k=x_0]$ ($k \ge 3$) such that (with indices taken $\modulo k$)
$e_j=[x_j\,,x_{j+1}] \in E_X$ and $x_{j+1} \ne x_{j-1}$ for all~$j$. 
The \emph{defect} of an edge $e \in E_X$ with respect to $C$ is 
$$
\df_C(e) = \sum_{j: e_j=e} (-1)^j \qquad (=0\; \text{ if }\; e \notin E_C).
$$
The path $C$ is said to have non-zero 
defect if $\df_C(e) \ne 0$ for some $e \in E_C\,$.
\end{dfn}

Every cycle is a  nonbacktracking closed path. 
In general, $C$ can cross an edge several times.
If we shift the vertex numbering along $C$ by $1$ modulo $k$, 
then the defect changes sign, but the absolute value remains unchanged. 

\begin{lem}\label{lem:even}
Let $\lambda \in \spec(Q_{V^o})$ and according to 
Proposition \ref{pro:specQVo}, let $f$ be an associated eigenfunction of $Q_{V^o}$
with $Qf(a) = Qf(b) = 1$, so that $f = f_{\nu_o}$ of case \emph{(II$^o$).} 

Then for every \emph{even} non-backtracking closed path 
$C = [x_0\,, x_1\,, \dots, x_k=x_0]$ ($k \ge 4$ even) with non-zero defect 
in the host graph $X$, the function $f$ can be used to construct a non-zero eigenfunction 
$g_C^*$ of $P_*$ 
with the same eigenvalue which is $=0$ on $X[V]$ outside of the vertices corresponding to 
$V^o$ substituted along the edges of the path. 

\smallskip

In particular, let $C_1\,,\dots, C_{\ell}$ is a fundamental cycle base. 
If $X$ is bipartite then the construction  
yields the $\ell = |E_X| - |X| +1$ linearly independent eigenfunctions $g_{C_i}^*$ of $P_*$.

\smallskip

If $X$ is not bipartite then the fundamental cycle base can be used to obtain $\ell - 1 = |E_X| - |X|$
linearly independent eigenfunctions of $P_*\,$; see the construction in the below proof.
\end{lem} 

\begin{proof} 
We define for any $e \in E_X$ and $v \in V^o$
$$
g_C^*(e,v) = \frac{\df_C(e)}{a_X(e)} f(v)\,,
$$
which is $\equiv 0$ when $e \notin E_C\,$, and $g_C^*(x)=0$ for every $x \in X \subset X[V]$.

Again, one needs to check that $g_C^*$ is balanced  only at the vertices of $C \subset X[V]$.
To understand this, we explain the construction in a different way: on the substituted copy of $V$ 
along each $e_j \in E_C$, we take $\frac{(-1)^j}{a_X(e_j)} f$, such that successive edges yield the 
required balance at their common vertex. If the path $C$ crosses the same edge several times,
then we sum the respective contributions, which leads to the stated values of $g_C^*(e,v)$.

\smallskip

Now let $T$ be a a spanning tree, and let $C_1\,,\dots, C_{\ell}$ be the associated fundamental cycle 
base, where each $e \in E_X \setminus E_T$ leads to one $C_i$ when added to $T$.

When $X$ is bipartite, then each of the cycles $C_i$ is even, 
and the (balanced) eigenfunctions $g_{C_i}^*$ are linearly independent like in the 
previous proof. 

When $X$ is not bipartite, then without loss of generality let $C_1\,,\dots, C_{\ell'}$ 
be the even ones
in the cycle base, and $C_{\ell'+1}\,,\dots, C_{\ell}$ the odd ones, where $\ell > \ell'$.
We first take the eigenfunctions $g_{C_i}^*$ for $i=1,\dots,\ell'$.
If $\ell' = \ell -1$ then we are done.

Otherwise, for $i=\ell'+1, \dots, \ell-1$, we take the shortest path $\pi = \pi_{i,\ell}$ 
in $T$ which connects some  vertex of $C_i$ with some vertex of $C_{\ell}$. (It may have length $0$.) 
Let $x_{0,i} \in C_i$ and $x_{0,\ell} \in C_{\ell}$ be the endpoints of that path. Then we construct 
a nonbacktracking path  $C_{i,\ell}$ of even length by starting at $x_{0,i}\,$, then 
going around $C_i$, 
then using $\pi$ to get to $x_{0,\ell}\,,$ then going around $C_{\ell}$ and finally back from 
$x_{0,\ell}$ to $x_{0,i}$ in the reversed direction. This path has non-zero defect since at 
least the defining edges of $C_i$ and $C_{\ell}$ (those who were added to $T$ to get the respective
cycle) are traversed only once. The functions $g_{C_i}^*$ ($i=1,\dots,\ell'$) and $g_{C_{i,\ell}}^*$ 
($i=\ell'+1, \dots, \ell-1$)  are linearly independent because each of them is non-zero on the 
substituted graph along the edge of $C_i$ which does not belong to $E_T\,$.
\end{proof}

The following case has escaped the attention of \cite{St}. See 
Example \ref{exa:mple4} in \S \ref{sec:examples}.

\begin{lem}\label{lem:mixed}
Let $\lambda \in \spec(Q_{V^o})$ have type \emph{(IV$^o$).}
Then for every pair of distinct edges $e = [x, y]\,, \bar e = [x, \bar y] \in E_X$ sharing an endpoint,
$f_{\nu_o}$ and $f_{\nu_o-1}$ 
can be used to construct an eigenfunction $g^*= g^*_{e,\bar e}$ of $P_*$  
with the same eigenvalue which is $=0$ on $X$ as well as on $X[V]$ outside of 
the vertices corresponding to $V^o$ substituted along the edges of $e$ and~$\bar e$.

\smallskip

This leads to $\sum_{x \in X} \bigl(\deg(x)-1\bigr) = 2|E_X| - |X|$ linearly independent
eigenfunctions.
\end{lem}

\begin{proof} For $v \in V^o$, we set 
$$
\begin{aligned}
g^*(e,v) &= \frac{1}{a_X(e)} \cdot \begin{cases} f_{\nu_o}(v)\,,&\text{if }\; x=e^a\,,\\
                               f_{\nu_o-1}(v)\,,&\text{if }\; x=e^b,
                              \end{cases}
 \quad \text{and}\\ 
g^*(\bar e,v) &= \frac{-1}{a_X(\bar e)}\cdot 
                          \begin{cases} f_{\nu_o}(v)\,,&\text{if }\; x=\bar e^a\,,\\
                               f_{\nu_o-1}(v)\,,&\text{if }\; x=\bar e^b.
                          \end{cases} 
\end{aligned}
$$
At all other vertices of $X[V]$, we set $g^* = 0$. This function has the stated properties.

\smallskip 

If $x \in X$ and $e_1\,,\dots, e_d \in E_X$ ($d = \deg(x)$)  are the edges of $X$ incident with $x$
then the functions $g^*_{e_j,e_d}$ are linearly independent for $j=1,\dots,d-1$, but their span
contains the other functions $g^*_{e_i,e_j}$ ($i, j < d, i \ne j$). The functions resulting in this 
way from different vertices $x \in X$ are clearly linearly independent. This leads to the stated
number. 
\end{proof}

\section{Completeness of the $\spec(Q_{V^o})$-eigenfunction extensions}\label{sec:complete}

In the last section we have studied how eigenfunctions of $Q_{V^o}$ can be use to construct
eigenfunctions of $P_*$ with the same eigenvalue. The latter are all \emph{nodal,}
that is, $\equiv 0$ on $X \subset X[V]$. If an eigenfunction of $P_*$ has the latter property 
then the associated eigenvalue must be in $\spec(Q_{V^o})$. 
In this section we show that all nodal eigenfunctions of $P_*$ 
arise from the functions provided by the extension 
lemmas \ref{lem:i}, \ref{lem:odd}, \ref{lem:even} and 
\ref{lem:mixed}. 
Thereby, we get the multiplicities of the respective eigenvalues. 

As a motivation, we start with a Lemma based on Proposition \ref{pro:G} combined with 
the considerations on poles and local spectra. 

\begin{lem}\label{lem:poles-spec}
Let $\lambda^* \in [-1\,,\,1]$.
\begin{itemize}
 \item[(a)] $\lambda^* \in \mathsf{S}_x(P_*)$ 
 for some $x \in X \subset X[V]$ if and only if
\begin{itemize} 
 \item[(a.1)] either $\varphi(\lambda^*) \in \spec(P)$ and $0 < |\psi(\lambda^*)| < \infty$
 \item[(a.2)] or $\psi(\lambda^*)=0$, in which case $\lambda^* \in \spec(Q_{V^o})$ or 
 $\theta(\lambda^*) = \lambda^*$.\\[-8pt]
\end{itemize}
 \item[(b)] If $\lambda^* \in \spec(P_*)$ does not satisfy \emph{(a)} then   
every $\lambda^*$-eigenfunction of $P_*$ is $\equiv 0$ on $X \subset X[V]$,
in which case $\lambda^* \in \spec(Q_{V^o})$.  
\end{itemize}
\end{lem}

\begin{proof} 
If $\lambda^* \in \spec(P_*)$ then the first possibility is that it is a pole of 
$G_*(x,x|z)$ for some $x \in X \subset X[V]$. In this case, by Propostion \ref{pro:G} and Lemma
\ref{lem:psine0},  $\varphi(\lambda^*) \in \spec(P)$ or else $\psi(\lambda^*)=0$. This proves (a).
Otherwise, Lemma \ref{lem:poles2} yields (b). 
\end{proof}

Now assume that $f^*: X[V] \to \R$ is a non-zero function which 
satisfies $P_*f^* = \lambda^* \, f$  and  $f^*|_X \equiv 0$, so that $\lambda^* \in  \spec(Q_{V^o})$.
Recall that for each $e \in E_X$, the function $f^e: V^o \to \R$ is 
defined by  $f^e(v) = f^*(e,v)$. It satisfies $Q_{V^o}f^e = \lambda^*\, f^e$, 
whence can be written uniquely as a linear combination
of the $\lambda^*$-eigenfunctions $f_1\,\dots, f_{\nu_o}$ of 
Proposition \ref{pro:specQVo} 
with $\nu_o=\nu_o(\lambda^*)$. Substract from this linear combination the part coming from
$f_1\,\dots, f_{\nu_o'}$ with $\nu_o'=\nu_o'(\lambda^*)$. 
The resulting function $f_{\rd}^e$ is 
\begin{enumerate}[leftmargin=1.7cm]
 \item[(1)] $\; \equiv 0$ if $\lambda^*$ is of type (I$^o$),
 \item[(2)+(3)] a multiple of $f_{\nu_o}$ if $\lambda^*$ is of
type (II$^o$) or (III$^o$), and 
 \item[(4)] a linear combination of $f_{\nu_o-1}$ and $f_{\nu_o}$ if $\lambda^*$ is of
type (IV$^o$).
\end{enumerate}
Putting together all the $f_{\rd}^e$ ($e \in E_X$) 
we get the function $f_{\rd}^*: X[V] \to \R$ with $f_{\rd}^*|_X \equiv 0$ 
and $f_{\rd}^*(e,v)=f_{\rd}^e(v)$ for $v \in V^o$. It again satisfies 
$P_*f_{\rd}^* = \lambda^*\, f_{\rd}^*$.
We call $f_{\rd}^*$ the \emph{reduced part} of~$f^*$.

\begin{thm}\label{thm:multiplicities}
 Assume that $f^*: X[V] \to \R$ is a non-zero function which 
satisfies $P_*f^* = \lambda^* \, f$  and  $f^*|_X \equiv 0$, so that $\lambda^* \in  \spec(Q_{V^o})$.
Let $f_{\rd}^*$ be its reduced part. Fix a spanning tree $T$ of $X$.
Then the following holds according to the type of the eigenvalue.\\[4pt]
\emph{(i)} If $\lambda^*$ is of type \emph{(I$^o$)} then $f_{\rd}^* \equiv 0$. 
\\[4pt]
\emph{(ii)} If $\lambda^*$ is of type \emph{(II$^o$)} then $f_{\rd}^*$ 
is a linear combination of the functions constructed in Lemma \ref{lem:even} via the spanning
tree $T$ of $X$, namely
\begin{itemize} 
\item if $X$ is bipartite, of the functions $g_{C_i}^*\,$, $i= 1,\dots,\ell$,
where $C_1\,,\dots, C_{\ell}$ is the fundamental cycle base of $X$ induced by $T$,  
\item if $X$ is not bipartite, of the functions $g_{C_i}^*\,$, $i= 1,\dots,\ell'$ and 
$g_{C_{i,\ell}}^*\,$, $i= \ell'+1,\dots,\ell-1$ constructed in the proof of Lemma \ref{lem:even} with 
respect to $T$.
\end{itemize}
\emph{(iii)} If $\lambda^*$ is of type \emph{(III$^o$)} then $f_{\rd}^*$ 
is a linear combination of the functions $f_{C_i}^*\,$, $i= 1,\dots,\ell$, constructed in 
Lemma \ref{lem:odd} with respect to  $T$.
\\[4pt]
\emph{(iv)} If $\lambda^*$ is of type \emph{(IV$^o$)} then $f_{\rd}^*$ 
is a linear combination of the functions $g^*_{e,\bar e}\,$, constructed in 
Lemma \ref{lem:mixed}, where $e,\bar e \in E_X$ are distinct 
edges sharing an endvertex in $X$.
\end{thm}

\begin{proof}
(i) is clear.
\\[4pt]
(ii) Let $E_X \setminus E_T = \{e_1\,,\dots, e_{\ell}\}$, so that $C_i$ is the cycle obtained when
$e_i$ is added to $E_T\,$. Note that when $i \ne j$ then $e_j$ is not an edge of $C_i$. 
For each $i$, the restriction of $f_{\rd}^*$ to the copy of 
$V^o$ substituted along $e_i$ (as constructed above) is of the form 
$$
f_{\rd}^{e_i} = \frac{\alpha_i}{a_X(e_i)} \,f_{\nu_o}\,.
$$
If $X$ is bipartite then consider the function
$$
h^* = f_{\rd}^* - \sum_{i=1}^{\ell}  \alpha_i\, g_{C_i}^*
$$
Again, it satisfies $P_*h^* = \lambda^*\, h^*$, it is $\equiv 0$ on $X \subset X[V]$, it is reduced, 
and most importantly, $h^{e_i} \equiv 0$ for each $i$. That is, it lives on 
$T[V]$, and along each $e \in E_T\,$ we have that $h^e = \beta_e \, f_{\nu}$ for some
$\beta_e \in \R$. Since $h^*$ has value $0$ at each leaf of $T$, we must have $\beta_e = 0$ for
all edges incident with a leaf. Successively working inwards from the leaves, we get that $h^*\equiv 0$.

If $X$ is not bipartite then define the function $h^*$ as 
$$
h^* = f_{\rd} - \sum_{i=1}^{\ell'}  \alpha_i\, g_{C_i}^* 
- \sum_{i=\ell'+1}^{\ell-1} \alpha_i\, g_{C_{i,\ell}}^* 
$$
(The second part is empty when $\ell' = \ell-1$). Arguing as above, $h^*$ is $0$ outside of the subgraph 
$X'[V]$, where $X'$ is obtained from $T$ by adding the single edge $e_{\ell}$. Now $X'$
consists of the odd cycle $C_{\ell}$ plus pieces of $T$ attached to its vertices. Again, on the 
substitution of $V$ along each edge  $e \in E_{X'} = E_T \cup \{e_\ell\}$, the function $h^e$
is a multiple of $f_{\nu_o}\,$. By the same pruning procedure as above, it must be $\equiv 0$
when $e$ is an edge of one of the subtrees of $T$ attached 
to $C_{\ell}\,$, so that $h^*$ can be $\ne 0$ only on 
$C_{\ell}[V]$. But this also cannot happen since the cycle is odd.
\\[4pt]
(iii) This is completely analogous to the proof of (ii) in the bipartite case, and left to the reader.
\\[4pt]
(iv) We use once more the tree $T$ with $E_X \setminus E_T = \{e_1\,,\dots, e_{\ell}\}$. Let 
$e_i = [x_i=e_i^a\,,y_i=e_i^b]$ and let $\bar e_i$ and $\tilde e_i$ edges of $T$, such that $x_i$ 
is an endvertex of $\bar e_i$ and $y_i$ an endvertex of $\tilde e_i$.
In the substitution along the edge $e_i$, we have
$$
f_{\rd}^{e_i} = \frac{\alpha_i}{a_X(e_i)} f_{\nu} +  \frac{\beta_i}{a_X(e_i)}f_{\nu-1} 
$$
for some $\alpha_i\,, \beta_i \in \R$. Using Lemma \ref{lem:mixed}, we set 
$$
h^* = f_{\rd} - \sum_{i=1}^{\ell} 
\Bigl(\alpha_i \,g^*_{e_i,\bar e_i} - \beta_i \,g^*_{e_i,\tilde e_i}\Bigr).  
$$
Once more, it satisfies $P_*h^* = \lambda^*\, h^*$, it is $\equiv 0$ on $X \subset X[V]$, it is reduced, 
and it lives on $T[V]$. Along each $e \in E_T\,$ we have that 
$h^e = \frac{\alpha_e}{a_X(e)} \, f_{\nu} + \frac{\beta_e}{a_X(e)} \, f_{\nu-1}$ for some
$\alpha_e\,, \beta_e \in \R$. 

Let $y$ be a leaf of $T$ and $e=[x,y]$ be the edge terminating in $y$. We choose another edge
$\bar e$ which is incident with $x$.

If $x=e^a$ then must have $Qh^e(b)=0$ whence $\beta_e =0$, and we construct the function 
$h_1^* = h^* - \alpha_e\,g^*_{e,\bar e}\,$. 

If $x=e^b$ then must have $Qh^e(a)=0$ whence $\alpha_e =0$, and we let 
$h_1^* = h^* - \beta_e\,g^*_{e,\bar e}\,$. 

In both cases, $h_1$ satisfies $P_*h_1^* = \lambda^*\, h_1^*$, 
it is reduced, and it lives on the tree $T_1[V]$, where $T_1$ is obtained from $T$ by deleting
$e$ and $y$. We can now continue to work inwards from the leaves, removing an edge at each step.
This shows that $h^*$ and hence also $f_{\rd}^*$ are linear combinations of the $g^*_{e,\bar e}\,$, 
$e \in E_X\,$.
\end{proof}

\begin{dfn}\label{def:S1S2}
 $$
\begin{aligned}
\mathcal{R} &= \bigl\{ \lambda^* \in [-1\,,\,1] \setminus \spec(Q_{V^o}) : 
 \psi(\lambda^*)\ne 0\,,\; \varphi(\lambda^*) \in \spec(P) \bigr\} \quad \text{and}\\
\mathcal{S} &= \bigl\{ \lambda^* \in [-1\,,\,1] \setminus \spec(Q_{V^o}) : 
 \psi(\lambda^*) = 0\,,\; \theta(\lambda^*)=\lambda^*  \bigr\}.
\end{aligned}
$$
\end{dfn}

Here, $\mathcal{R}$ stands for ``response'' (to the natural frequencies), 
and $\mathcal{S}$ stands for ``special''.
We briefly explain how the eigenvalues of $Q$ are related with the 
three sets $\mathcal{R}\,$, $\mathcal{S}$ and $\spec(Q_{V^o})$. Recall Proposition
\ref{pro:specQ}.

\begin{lem}\label{lem:compare} 
\emph{(a)} If $\lambda^* \in \spec(Q)$ is such that $\nu_Q'(\lambda^*) \ge 1$, in particular, 
if $\lambda^*$ has type \emph{(I),} then  $\lambda^* \in \spec(Q_{V^o})$.
\\[4pt]
\emph{(b)} If $\lambda^* \in \spec(Q)$ has type \emph{(II)} then $\lambda^* \in \spec(Q_{V^o})$
or else $\psi(\lambda^*) \ne 0$ and $\varphi(\lambda^*)=1 \in \spec(P)$, 
whence $\lambda^* \in \mathcal{R}\,.$

Conversely, if $\lambda^* \in [-1\,,\,1]$ satisfies $\varphi(\lambda^*)=1$ then 
$\lambda^* \in \spec(Q_{V^o})$, or else $\psi(\lambda^*) \ne 0$ and $\lambda^* \in \spec(Q)$
has type \emph{(II)} with $\nu_Q(\lambda^*) = 1$.
\\[4pt]
\emph{(c)} If $\lambda^* \in \spec(Q)$ has type \emph{(III)} then $\lambda^* \in \spec(Q_{V^o})$,
or else $\psi(\lambda^*) \ne 0$ and $\varphi(\lambda^*)=-1$, which is
in $\spec(P)$ if and only if $X$ is bipartite, 
and in this case $\lambda^* \in \mathcal{R}\,.$

Conversely, if $\lambda^* \in [-1\,,\,1]$ satisfies $\varphi(\lambda^*)=-1$ then 
$\lambda^* \in \spec(Q_{V^o})$, or else $\psi(\lambda^*) \ne 0$ and $\lambda^* \in \spec(Q)$
has type \emph{(III)} with $\nu_Q(\lambda^*) = 1$.
\\[4pt]
\emph{(d)} If $\lambda^* \in \spec(Q)$ is of type \emph{(IV)} then
$\lambda^* \in \spec(Q_{V\!-b})$. Either $\lambda^* \in \spec(Q_{V^o})$,
or else $\lambda^* - \theta(\lambda^*) = \psi(\lambda^*) = 0$,
whence $\lambda^* \in \mathcal{S}\,$.
\end{lem}

\begin{proof}
(a) is clear from Proposition \ref{pro:specQ}.

\smallskip

Now suppose that $\lambda^* \notin \spec(Q_{V^o})$ and that it is not of type (I).
Suppose that $\psi(\lambda^*)=0$. We use again $X=\{x,y\}$ with the single edge $[x,y]$,
so that $X[V]=V$. We know that $\lambda^*$ is an eigenvalue of $P_*=Q$. By Theorem 
\ref{theorem:eigen}, we also must have $\theta(\lambda^*)=\lambda^*$. But then, again 
by Theorem \ref{theorem:eigen}, the associated eigenspace must have dimension 2,
which is only possible when $\lambda^*$ is of type (IV).

\smallskip 

(b) If $\lambda^* \notin \spec(Q_{V^o})$ then, again by Proposition \ref{pro:specQ},  
$\nu_Q'(\lambda^*)=0$ and $\nu_Q(\lambda^*)=1$. By the above reasoning, we must have 
$\psi(\lambda^*) \ne 0$, and by Lemma \ref{lem:i}, we get a $\lambda^*$-eigenfunction 
of $P_*$ which extends the constant function $\uno$ on $X$. Therefore, by Theorem \ref{theorem:eigen},
$\varphi(\lambda^*)$ is the eigenvalue of $P$ of the eigenfunction $\uno_X\,$.

\smallskip

(c) is completeley analogous. We may first use the bipartite host graph $X = \{x,y\}$ as above with 
the eigenfunction $\uno_x - \uno_y$ to conclude that $\varphi(\lambda^*)=-1$.

\smallskip

The converse statements of (b) and (c) are now clear.

\smallskip 
(d) follows from the considerations preceding the proof of (b)  plus 
Theorem~\ref{theorem:eigen}.
\end{proof}

\begin{cor}\label{cor:S2}
$\qquad \mathcal{S} = \{ \lambda^* \in \spec(Q) \setminus \spec(Q_{V^o}):
\nu_Q(\lambda) = 2\},$\\[4pt]
and those eigenvalues of $Q$ have type \emph{(IV)}. 
\end{cor}

It is important to have present that $\mathcal{S}$ may be empty.

\section{Description of $\spec(P_*)$}\label{sec:spec}

Summarising the considerations of the previous sections, we are now able to provide the 
desired description of $\spec(P_*)$ in terms of the input data coming from $(X,P)$ and $(V,Q)$. 
Below, $\mathcal{N}$ stands for the nodal part of the spectrum.
\begin{thm}\label{theorem:specP*} 
As a set (i.e., without counting multiplicities) one  has
$$
\spec(P_*) = \mathcal{R} \cup \mathcal{S} \cup \mathcal{N} \,,\quad
\text{where} \quad \mathcal{N} = \spec(Q_{V^o}) \setminus \mathsf{Exc} \,.
$$
The exceptional set $\mathsf{Exc}$ is empty except possibly in one of the following special 
cases.
\\[4pt]
\emph{(A)}  If $X$ is a tree then 
$$
\mathsf{Exc} = \bigl\{ \lambda^* \in \spec(Q_{V^o}) \setminus \spec(Q_V) : 
\nu_o(\lambda^*) = 1 \; \text{and}\; \lambda^* \; \text{has type {\rm (II$^o$)} 
or {\rm (III$^o$)}}\bigr\}
$$
\emph{(B)} If $X$ has precisely one cycle and it length is odd then
$$
\begin{aligned}
&\quad\mathsf{Exc} = \bigl\{ \lambda^* \in \spec(Q_{V^o}) \setminus \spec(Q_V) : 
\nu_o(\lambda^*) = 1 \; \text{and}\; \lambda^* \; \text{has type {\rm (II$^o$)}}\bigr\}\\
&\cup \bigl\{ \lambda^* \in \spec(Q_{V^o}) \cap \spec(Q_V) : 
\nu_o(\lambda^*) = \nu_Q(\lambda)= 1 \; \text{and}\; \lambda^* \; 
\text{has types {\rm (II$^o$)} and {\rm (III)}}\bigr\}.
\end{aligned}
$$
Furthermore, the following holds.
\\[4pt]
\emph{(C)} If $X$ is bipartite then 
$\quad\spec(P_*) \supset \spec(Q)$.\\[4pt]
\emph{(D)} If $X$ is not bipartite then\\[2pt]
\hspace*{2.5cm} $\spec(P_*) \supset \spec(Q) \setminus \{ \lambda \in \spec(Q) : 
\lambda\; \text{ is of type \emph{(III)} and } \; \nu_Q(\lambda)=1 \}$.
\end{thm}

\begin{proof}
This follows from Theorem \ref{theorem:eigen}, lemmas \ref{lem:i}, \ref{lem:odd}, \ref{lem:even} 
and \ref{lem:mixed} as well as Theorem \ref{thm:multiplicities}; see also Corollary \ref{cor:S2}.
Statements (C) and (D) follow from lemmas \ref{lem:ii}, \ref{lem:iii}, \ref{lem:iv}
and \ref{lem:compare}.
\end{proof}

More involved are the eigenvalue multiplicities which we summarise next. 
Recall the notation displayed at the beginning of \S \ref{sec:Green}. 
To types (I)--(IV) of $\lambda^* \in \spec(Q)$ it is convenient to add type (0) if 
$\lambda^* \notin \spec(Q)$. Part of the following proposition is 
covered by \cite[Lemma 3.2]{St}.

\emph{Here and in the sequel, we set $\delta_{\rm bi} =1$ when $X$ is bipartite, 
$\delta_{\rm bi}=0$ otherwise.}

\begin{pro}\label{pro:mult}
\quad\emph{(1)} If $\lambda^* \in \mathcal{R}$ then 
 $\;\nu_*(\lambda^*) = \nu_P\bigl(\varphi(\lambda^*)\bigr)$.\\[4pt]
 \emph{(2)} If $\lambda^* \in \mathcal{S}$ then  $\;\nu_*(\lambda^*) = |X|$.\\[4pt]
 \emph{(3)} If $\lambda^* \in \spec(Q_{V^o})$ then the following table shows
  $\nu_*(\lambda^*)$ according to the possible types. 
\begin{center}
{\footnotesize
\begin{tabular}{ c |c c c c }
          & \rm (I$^o$) & \rm (II$^o$) & \rm (III$^o$) & \rm (IV$^o$) \\[2pt]
        \hline\\[-7pt]
\rm (0)\!\! & -- & $|E_X|-|X|+\delta_{\rm bi}$ & $|E_X|-|X|+1$ & $2|E_X|-|X|$
\\[4pt]
\rm (I)\!\! &  $\nu_o(\lambda^*)|E_X|$ & $\nu_o(\lambda^*)|E_X|-|X|+\delta_{\rm bi}$ & 
  $\nu_o(\lambda^*)|E_X|-|X|+1$ & $\nu_o(\lambda^*)|E_X|-|X|$
\\[4pt]
\rm (II)\!\! & $\nu_o(\lambda^*)|E_X|+1$  & $\nu_o(\lambda^*)|E_X|-|X|+1+\delta_{\rm bi}$ & 
  $\nu_o(\lambda^*)|E_X|-|X|+2$ & $\nu_o(\lambda^*)|E_X|-|X|+1$
\\[4pt]
\rm (III)\!\! & $\nu_o(\lambda^*)|E_X|+\delta_{\rm bi}$ & $\nu_o(\lambda^*)|E_X|-|X|+2\delta_{\rm bi}$ & 
   $\nu_o(\lambda^*)|E_X|-|X|+1+\delta_{\rm bi}$ & $\nu_o(\lambda^*)|E_X|-|X|+\delta_{\rm bi}$  
   \\[4pt]
\rm (IV)\!\! & $\nu_o(\lambda^*)|E_X|+|X|$ & $\nu_o(\lambda^*)|E_X|+ \delta_{\rm bi}$& 
$\nu_o(\lambda^*)|E_X|+ 1$ & $\nu_o(\lambda^*)|E_X|$ \\  
        \end{tabular}  
}
\end{center}
\end{pro}

Note the following.
\begin{itemize}
 \item When $\nu_o(\lambda^*)=1$, type (I) can only occur in combination with type (I$^o$).
 \item When $X$ is a tree and $\nu_o(\lambda^*)=1$, the entries in (II$^o$) \& (0) and in 
 (III$^o$) \& (0) are $=0$, which confirms statement (A) of Theorem \ref{theorem:specP*}.
 \item When $X$ has precisely one cycle, and the latter has odd length (i.e. $X$ is a tree augmented by
 one edge giving rise to an odd cycle) and $\nu_o(\lambda^*)=1$, the entries in (II$^o$) \& (0) and in 
 (III$^o$) \& (III) are $=0$, which confirms statement (B) of Theorem~\ref{theorem:specP*}.
\end{itemize}
Thus, the numbers in the table also cover the exceptional cases.

\begin{rmk}\label{rmk:nodal}
To each $\lambda^* \in \mathcal{N}$ there correspond eigenfunctions which are indeed nodal in the
sense that they vanish on $X \subset X[V]$. However, it may occur that there are additional 
eigenfunctions which do not have this property, namely when 
$\lambda^* \in \spec(Q_{V}) \cap \spec(Q_{V^o})$; see examples \ref{exa:mple1} and \ref{exa:mple2}.
\end{rmk}

\medskip
\goodbreak

\textbf{The spectral gap}

\smallskip

For a stochastic, reversible Markov chain transition matrix such as $P$ on $X$, the 
\emph{spectral gap} is $1- \lambda_1\,$, 
where $\lambda_0=1$ and $\lambda_1$ are the largest and second 
largest eigenvalues of $P$, respectively. The name is related to (minus) the Laplacian $I-P$, 
where the spectral gap is the distance of its second eigenvalue 
from the bottom $0$ of its spectrum. 
It is a well-known basic consequence of the spectral representation 
of $P$ that this is a key quantity for understanding the speed of convergence of the 
Markov chain distributions (the rows of $P^n$) to equilibrium
(the normalised measure $m_X$). 

We now address the question whether for $P_*$ on $X[V]$, the second largest eigenvalue
has to be a solution of the equation $\varphi(\lambda_1^*) = \lambda_1\,$, or whether
the ``nodal'' case $\lambda_1^* \in \spec(Q_{V^o})$ may arise.
For this, it is natural to assume that $X$ has at least three vertices, in which case it
is an easy exercise to show that $|\lambda_1| < 1$. Recall that we assume that $V \supset \{a,b\}$
strictly. The following is a corollary to Proposition \ref{pro:phiprop}.

\begin{cor}\label{cor:gap}
If $|X| \ge 3$ and $V^o$ is connected then $\lambda_1^* \notin \spec(Q_{V^o})$, and it is 
the largest solution of the equation $\varphi(\lambda^*) = \lambda_1\,$. 
In particular, $\lambda_1^* \in \mathcal{R}\,$.

Furthermore, if $\lambda_1 \ge 0$ the statement remains true without assuming connectedness of
$V^o$.
\end{cor}

\begin{proof}
When $V^o$ is connected, $\varphi(z)$ is continuous in the interval
$[\lambda_0(Q_{V_o})\,,\,1]$ with $\varphi(1)=1$ and $\varphi\bigl(\lambda_0(Q_{V^o})\bigr) =  -1$,
see Proposition \ref{pro:phiprop}. Since $|\lambda_1| <1$ there must be solutions of 
the equation $\varphi(\lambda^*) = \lambda_1\,$ in the open interval
$\bigl(\lambda_0(Q_{V_o})\,,\,1\bigr)$. We take the largest solution. This must be $\lambda_1^*\,$.
It is not contained in $\spec(Q_{V^o})$, being larger than the spectral radius of $Q_{V^o}\,$. 
And in that interval, $\psi(z) > 0$, concluding the proof. 

If $\lambda_1 \ge 0$ then by Proposition \ref{pro:phiprop}, the equation has a unique solution 
in $[\lambda_0(Q_{V\!-b})\,,\,1]$, and this is $\lambda_1^*\,$.
\end{proof}

Note that when Corollary \ref{cor:gap} applies, even beyond the assumption of
connectedness of $V^o$, one gets $\lambda_1^* > 0$, and further substitution
does no more require that assumption.

\section{Examples}\label{sec:examples}

\medskip

We now provide some examples, starting with  the 
simple case of Fig.~1.

\begin{exa}\label{exa:mple0} In Figure~1,
$X = \{x_0\,,\dots, x_4 \}$ is the 5-cycle, neighbourhood $x_k \sim x_{k+1}$ modulo $5$, and 
we take for $P$ the simple random walk (transition probabilities between neighbours 
are $=1/2$). We have 
$$
\spec(P) = \biggl\{ \lambda_0 = 1\,, \lambda_1= \cos \frac{2\pi}{5}=\frac{\sqrt{5}-1}{4}\,, 
\lambda_2 = \cos \frac{4\pi}{5}=-\frac{\sqrt{5}+1}{4} \biggr\}\,,
$$
and the two eigenvalues $\ne 1$ have multiplicity 2 each.
The respective eigenspaces are spanned by the functions
$$
f_0 \equiv 1\,, \quad f_j(x_k) = \cos \frac{jk\pi}{5}\,,\quad g_j(x_k) = \sin \frac{jk\pi}{5}\,,
$$
where the index $j$ refers to the eigenvalue $\lambda_j\,$.

\smallskip

The graph $V = \{a,u, b, v\}$ is the 4-cycle with the additional edge $V^o = [u,v]$, and we 
also take simple random walk for $Q$. 

We have $\spec(Q) = \bigl\{ 1 \text{\; [type (II)]} \,,0\text{\; [type (III)]}\,, 
-1/3\text{\; [type (I)]}\,, -2/3\text{\; [type (II)]} 
\bigr\}\,$, each with multiplicity $1$. 
Next, $\spec(Q_{V^o}) = \bigl\{ 1/3\text{\; [type (II$^o$)]}, -1/3\text{ [type\; (I$^o$)]}\bigr\}\,$, 
again each with multiplicity $1$.

For the computation of $\varphi(z)$, $\psi(z)$ and $\theta(z)$, we first compute 
$\mathcal{G}_{V^o}(z)$ and use \eqref{eq:psi-theta}, then apply Lemma \ref{lem:Fab}:
$$
\psi(z) = \theta(z) = \frac{1}{3z-1}\,, \; \varphi(z) = 3z^2 - z -1.
$$
For the application of Lemma \ref{lem:bvalue} (left to the reader), we note that also 
$F_{V\!-b}(u,a|z) = F_{V\!-b}(v,a|z) = F_{V\!-a}(u,b|z) = F_{V\!-a}(v,b|z) = 1/(3z-1).$

We consider the equations $\varphi(\lambda^*) = \lambda \in \spec(P)\,$ (in the numbering, we
anticipate the descending order of the eigenvalues of $P_*$):
$$
\varphi(\lambda^*)= 1 \implies 
\lambda_0^* = 1\,,\quad \lambda_6^* = -\dfrac{2}{3}\,.
$$
Both belong to $\spec(Q) \setminus \spec(Q_{V^o})$, in accordance with Lemma \ref{lem:compare}.
The multiplicities with respect to $P_*$ are $1$, with respective unique extensions of the eigenfunction
$f_0$ of $P$ according to Lemma \ref{lem:bvalue} and 
Theorem \ref{theorem:eigen}.
$$
\varphi(\lambda^*)= \frac{\sqrt{5}-1}{4} \implies 
\lambda_1^* = \dfrac{1+\sqrt{10 + 3\sqrt{5}}}{6}\,,\quad 
\lambda_5^* = \dfrac{1-\sqrt{10 + 3\sqrt{5}}}{6}\,,
$$
Their multiplicities with respect to $P_*$ are $2$, with respective linearly independent 
eigenfunctions that extend $f_1$ and $g_1\,$. Analogously, 
$$
\varphi(\lambda^*)= -\frac{\sqrt{5}+1}{4} \implies 
\lambda_2^* = \dfrac{1+\sqrt{10 - 3\sqrt{5}}}{6}\,,\quad 
\lambda_3^* = \dfrac{1-\sqrt{10 - 3\sqrt{5}}}{6}\,,
$$
and their multiplicities with respect to $P_*$ are also $2$, with respective linearly independent 
eigenfunctions that extend $f_2$ and $g_2\,$. 
We omit the straightforward computation of those extensions.
We have thus completed the computation of $\mathcal{R}\,$, and $\mathcal{S}$ is empty.

Coming to $\spec(Q_{V^o})$, we get from Theorem \ref{theorem:specP*} (B) that 
$\mathsf{Exc} = \{ 1/3\}$, while the eigenvalue $\lambda_4^* = -1/3$ becomes 
part of $\spec(P_*)$ with corresponding
multiplicty $|E_X| = 5$. 
We also note that $0 \in \spec(Q)$ is not contained in $\spec(P_*)$; see  
Theorem \ref{theorem:specP*} (D). We get
$$
\mathcal{R} \cup \mathcal{N} = \biggl\{ 1\;\ (\nu_*=1)\,,\;  -\dfrac{2}{3}\;\ (\nu_*=1)\,,\;
        \dfrac{1 \pm \sqrt{10 \pm 3\sqrt{5}}}{6}\;\ (\nu_*= 2\; \text{for each}) \biggr\} \cup 
        \biggl\{-\dfrac{1}{3}\;\ (\nu_*=5)\biggr\}.
$$
\end{exa}

We now work out how to handle three basic general expamples of candidate graphs $V$ for 
substitution, without specifying the host graph $X$. 

\begin{exa}\label{exa:mple1} \emph{Path of length $L$.}
\\[4pt]
Here $V = \{ v_0\,,v_1\,,\dots, v_L\}$ with 
$v_{j-1} \sim v_j$ ($j=1,\dots, L$) is the path
with length $L \ge 2$, and $Q$ is SRW. 
Adapting the computations of \cite[Lemmas 5.3 and 9.87(c)]{WMarkov} to the replacement 
$z \leftrightarrow 1/z$, we get 
$$
\varphi(z)=\mathcal{T}_L(z)\,,\quad \psi(z) = \frac{1}{\mathcal{U}_{L-1}(z)}\,,\AND
\theta(z) = z - \frac{\mathcal{T}_L(z)}{\mathcal{U}_{L-1}(z)}\,,
$$
where $\mathcal{T}_L$ and $\mathcal{U}_{L}$ are the $L^\text{th}$ Chebyshev polynomials of
the first, resp. second kind. 
We have 
$$
\spec(Q) = \bigl\{ \lambda_k = \cos \tfrac{k\pi}{L}\,,\; k=0,\dots, L \bigr\}. 
$$
The eigenfunction associated with 
$\lambda_k$ is $f_k(v_j) = \cos \tfrac{jk\pi}{L}$,
with $\nu_Q(\lambda_k)=1$ and $\nu_Q'(\lambda_k)=0$. Furthermore, 
$\lambda_k$ is of type (II) when $k$ is even and of type (III) when $k$ is odd. 
The set $\mathcal{S}$ is empty.

\smallskip

Note that $Q_{V^o}$ is $\frac{1}{2}$ times the adjacency matrix of the path  
$v_1 \sim v_2 \sim \dots \sim v_{L-1}$ of length $L-2$. 
$$
\spec(Q_{V^o}) = \{ \lambda_k : k = 1, \dots, L-1 \} \subset \spec(Q),
$$
but associated with $\lambda_k$ we have a different eigenfunction, namely 
$g_k(v_j) = \sin \tfrac{jk\pi}{L} \big/ \sin \tfrac{k\pi}{L}$ for $k=1,\dots, L-1$.
We have $\nu_o(\lambda_k)=1$ and $\nu_o'(\lambda_k)=0$, and 
$\lambda_k$ is of type (II$^o$) when $k$ is odd and of type (III$^o$) when $k$ is even.

\smallskip

For the rest of these computations, we assume that \emph{$L$ is even.} The case when $L$ is odd
will then be clear. 

\smallskip

First, we consider the equation $\varphi(\lambda^*) = \lambda$, where $\lambda \in [-1\,,\,1]$ is
given (a candidate eigenvalue of $P$), and we want $\lambda^* \notin \spec(Q_{V^o})$. We write
$$
\lambda = \cos \alpha \AND \lambda^* = \cos \alpha^* \quad \text{with }\; 
\alpha, \alpha^* \in [0\,,\,\pi].
$$
Then $\varphi(\lambda^*) = \cos L\alpha^*$, and for given $\alpha$ we want to solve 
$\cos L\alpha^* = \cos \alpha$ in $\alpha^*$.

First, let $\lambda = 1$, i.e., $\alpha =0$. Of the solutions, only
$\alpha^* = 0$ and $\alpha^*=\pi$ are such that 
$\lambda^* = \pm 1 \notin \spec(Q_{V^o})$. The remaining ones are 
$$
\lambda^* \in \{ \lambda_{2k}  : k = 1, \dots, \frac{L}{2}-1 \}
\subset \spec(Q_{V^o}).
$$
Next, let $\lambda = -1$, i.e., $\alpha =\pi$. We find the solutions
$$
\lambda^* \in \{ \lambda_{2k-1} : k = 1, \dots, \frac{L}{2} \biggr\}
\subset \spec(Q_{V^o}).
$$
Finally, when $\lambda \ne \pm 1$, i.e., $\alpha \notin \{ 0, \pi\}$,
we find $L$ distinct solutions for $\alpha^*$, namely
\begin{equation}\label{eq:alpha*}
\alpha^* \in \biggl\{ \frac{\alpha + 2(\ell-2)\pi}{L}\,,\;\frac{2\ell\pi - \alpha}{L} 
  : \ell = 1, \dots, \frac{L}{2} \biggr\}.
\end{equation}
Now suppose that for the host graph $X$ and its transition matrix $P$, we have
an eigenfunction $f$ such that $Pf = \lambda \, f$ with $\lambda \ne \pm 1$. 
Let $\lambda^* = \cos \alpha^*$ with $\alpha^*$ as in \eqref{eq:alpha*}.
Then, based on Lemma \ref{lem:bvalue}, we can compute the extension $f^*$ of $f$ 
which satisfies $P_*f^* = \lambda^* \, f^*$ explictly.
Namely (recalling that $a = v_0$ and $b=v_L$), 
by \cite[Lemma 5.5 and Example 5.6]{WMarkov}, for $j=1,\dots, L-1$, 
$$
F_{V\!-b}(v_j,a|\lambda^*) = \frac{\mathcal{U}_{L-j-1}(\lambda^*)}{\mathcal{U}_{L-1}(\lambda^*)} =
\frac{\sin (L-j)\alpha^*}{\sin L\alpha^*}
\!\!\AND\!\!
F_{V\!-a}(v_j,b|\lambda^*) = \frac{\mathcal{U}_{j-1}(\lambda^*)}{\mathcal{U}_{L-1}(\lambda^*)} =
\frac{\sin j\alpha^*}{\sin L\alpha^*}, 
$$
so that for any $e \in E_X$
\begin{equation}\label{eq:f*evj}
f^*(e,v_j) = f(e^a) \,\frac{\sin (L-j)\alpha^*}{\sin L\alpha^*} + 
f(e^b) \, \frac{\sin j\alpha^*}{\sin L\alpha^*},
\end{equation}
which can of course be checked directly. Recall that the multiplicity $\nu_P(\lambda)$ of 
$\lambda \in \spec(P) \setminus \{\pm 1 \}$ coincides with the multiplicty of 
$\lambda^* = \cos \alpha^*$ as an eigenvalue of $P_*$.

We come back to $\lambda =  1$. Since $L$ is even, $X[V]$ is bipartite. The constant function $\uno_X$ 
on $X$ extends on one hand to the constant function $\uno_{X[V]}$ with eigenvalue $\lambda^*=1$.
It also extends to the function which is $\equiv 1$ on the bipartite class of $X$ in $X[V]$, 
and $\equiv -1$ on the other bipartite class of $X[V]$, with eigenvalue $\lambda^* = -1$.

Note that with $\delta_{\rm bi}$ as in Proposition \ref{pro:mult},
$$
\sum_{\lambda \in \spec(P) \setminus \{\pm 1\}} \nu_P(\lambda) = |X| -1 - \delta_{\rm bi}
$$
We conclude that the sum of the multiplicites of the eigenvalues  of $P_*$ which are 
in the ``response'' set $\mathcal{R}$ is 
$$
2 + \bigl( |X| -1 - \delta_{\rm bi} \bigr)L
$$
On the other hand, via the table of Proposition \ref{pro:mult}, 
the sum of the multiplicites of the eigenvalues  
of $P_*$ which belong to the ``nodal'' set $\mathcal{N}$ is
$$
\bigl( |E_X| - |X| + 2\delta_{\rm bi} \bigr)\frac{L}{2} + \bigl( |E_X| - |X| + 2\bigr)\frac{L-2}{2}.
$$
This is an example where elements of $\mathcal{N}$ also correspond to 
eigenfuncions which are $\ne 0$ on $X$, compare with Remark \ref{rmk:nodal}. 
Finally, $\mathcal{S}= \emptyset$.
Taking into account that $|L| = |V|-1$, we get that the total sum of all multiplicites
is $|X| + |E_X|\big(|V|-2\bigr) = \big|X[V]\big|$, as it must be.

Note that the specific case when $L=2$ yields the well-known 
\emph{edge-subdivision} of $X$. 
\end{exa}

\begin{exa}\label{exa:mple2} \emph{Circle of length $M = 2L$.}
\\[4pt]
Let $V = \{ v_0\,,v_1,\dots, v_{M-1} \}$ with $v_{j} \sim v_{j+1}$, 
where indices are taken modulo $M$,
and $Q$ is SRW: $q(v_j\,,v_{j\pm 1}) = 1/2$ for $j= 0, \dots, M-1$. We choose $a= v_0$ 
and $b=v_L\,$.

Note that SRW on the path of length $L$, as considered in Example \ref{exa:mple1}, 
is the factor chain\footnote{Factor chains correspond to what are called
equitable partitions in graph theory.} of the present $Q$ when uniting $v_j$ with $v_{M-j}$ 
for $j=1, \dots, L-1$.
Thus, again
$$
\varphi(z)=\mathcal{T}_L(z)\,,\quad \psi(z) = \frac{1}{\mathcal{U}_{L-1}(z)}\,,\AND
\theta(z) = z - \frac{\mathcal{T}_L(z)}{\mathcal{U}_{L-1}(z)}\,.
$$
Also,
$$
\spec(Q) = \bigl\{ \lambda_k = \cos \tfrac{k\pi}{L}\,,\; k=0,\dots, L \bigr\}, 
$$
as above, but here the multiplities are $\nu_Q(\lambda_k)=2$ for $k=1,\dots, L-1$, 
and for those indices, the normalised eigenfunctions are 
$$
f_{k,1}(v_j) = \sin \frac{jk\pi}{L} \Big/ \sin \frac{k\pi}{L} \AND
f_{k,2}(v_j) = \cos \frac{jk\pi}{L}\,,\quad j = 0,\dots, M-1. 
$$
Therefore $\nu_Q'(\lambda_k)= 1$, and the type of $\lambda_k$ with respect to $Q$ is 
(II) when $k$ is even and (III) when $k$ is odd.

\smallskip

The set $\mathcal{S}$ is again empty.

\smallskip

Next, $\nu_Q(\lambda_0) = \nu_Q(\lambda_L)=1$. 
The normalised eigenfunction associated with $\lambda_0=1$ is of course $\uno_V\,$, and the one associated 
with $\lambda_L = -1$ is $f(v_j)=(-1)^j$. In particular, $\lambda_0$ has type (II), and $\lambda_L$
has type (II) if $L$ is even, type (III) if $L$ is odd.

Note that $V^o$ is disconnected, consisting of two copies of the path of length $L-1$.
$$
\spec(Q_{V^o}) = \{ \lambda_k : k = 1, \dots, L-1 \} \subset \spec(Q),
$$
$\nu_o(\lambda_k)=2$, $\nu_o'(\lambda_k) = 1$, and 
normalised eigenfunctions are $f_{k,1}$ as above plus 
$$
g_{k,2}(v_j) = \begin{cases} 2\sin \frac{jk\pi}{L} \Big/ \sin \frac{k\pi}{L}\,,&j=1,\dots, L-1,\\
                             0\,, &j=L+1, \dots, M-1.   
               \end{cases}
$$
The type of $\lambda_k$ with respect to $Q_{V^o}$ is (II$^o$) when $k$ is odd and  (III$^o$) when 
$k$ is even.

\smallskip

Again, we now assume that \emph{$L$ is even,} so that $M$ is a multiple of $4$. 

\smallskip

The study of the equation $\varphi(\lambda^*) = \lambda$ with given $\lambda \in [-1\,,\,1]$ 
and solutions  $\lambda^* \notin \spec(Q_{V^o})$ proceeds precisely as in Example 
\ref{exa:mple1}.
The extension formula \eqref{eq:f*evj} remains unchanged for $j=1,\dots, L-1$, 
and $f^*(e,v_{M-j}) = f^*(e,v_j)$.

\smallskip

Given the host graph $X$,
the sum of the multiplicites of the eigenvalues  of $P_*$ which are 
in $\mathcal{R}$  is again
$$
2 + \bigl( |X| -1 - \delta_{\rm bi} \bigr)L\,,
$$
and via the above table, the sum of the multiplicites of the eigenvalues  
of $P_*$ which are in $\mathcal{N}$  is
$$
\bigl( 2|E_X| - |X| + 2\delta_{\rm bi} \bigr)\frac{L}{2} 
+ \bigl( 2|E_X| - |X| + 2\bigr)\frac{L-2}{2}.
$$
Again, there are also ``non-nodal'' associated eigenfunctions.
Takig into account that $2L = |V|$, we see also here that the total sum of all multiplicites
is $\big|X[V]\big|$, as it must be.
\end{exa}

\begin{exa}\label{exa:mple3} \emph{Another look at the circle of length $M = 2L$.}
\\[4pt]
We take $V$ and $Q$ as in Example \ref{exa:mple2}, but this time we choose $a= v_0$ and $b=v_{M-1}\,$.

Here we use the factor chain where $v_j$ and $v_{M-2-j}$ are merged into one state for 
$j=0,\dots, L-2$, while $v_{L-1}$ and $v_{M-1}=b$ remain alone. We get SRW on the path of
length $L$, and $F(a,b|z)$ is the same as $F(L-1,L|z)$ on this path. By \cite[Lemma 5.3]{WMarkov}, 
$$
\varphi(z) = \frac{\mathcal{T}_{L}(z)}{\mathcal{T}_{L-1}(z)}.
$$
Next, $\psi(z) = \frac12 + \frac12 F_{V\!-a}(v_1,b|z)$. 
For the latter, consider SRW on $\{0, \dots, M-1\}$ with state $0$ absorbing. Then 
$F_{V\!-a}(v_1,b|z) = F_{\text{SRW}}(1,M|z)$.   
Via the formula in the middle of page 119 in \cite{WMarkov},
this is equal to $1/\mathcal{U}_{M-2}(z)$ (recalling that we should replace $z \leftrightarrow 1/z$
when using \cite{WMarkov}), whence
$$
\psi(z) = \frac12 \biggl(1 + \frac{1}{\mathcal{U}_{M-2}(z)}\biggr)
$$
Finally
$\theta(z) = z U_{V\!-b}(a,a|z) = \frac{1}{2z}F_{v_1\,,v_0}(z)$, and analogously by 
\cite[Lemma 9.87(c)]{WMarkov}, 
$$
\frac{1}{z - F_{v_1\,,v_0}(z)} =  \frac{\mathcal{U}_{M-2}(z)}{\mathcal{T}_{M-1}(z)},
\quad \text{ whence } \quad
\theta(z) = \frac12 \biggl(z - \frac{\mathcal{T}_{M-1}(z)}{\mathcal{U}_{M-2}(z)}\biggr).
$$
The spectrum of $Q$ and the associated eigenfunctions are the same as in Example \ref{exa:mple2}, 
$\nu_Q(\lambda_0) = \nu_Q(\lambda_L)=1$, the type of $\lambda_0$ is (II) and the one of $\lambda_L$
with respect to the present choice of $b$ is (III). But now
the eigenfunctions associated with $\lambda_k$ for $k=1,\dots,L-1$ satisfy 
$f_{k,i}(b) \ne \pm f_{k,i}(a)$ for $i=1,2$, so that $\lambda_k$ has type (IV) and 
$\nu_Q(\lambda_k)=2$,
and one may verify directly that $\lambda_k - \theta(\lambda_k) = \psi(\lambda_k)=0$.
This time,
$$
\mathcal{S} = \{ \lambda_1\,,\dots, \lambda_{L-1} \}.
$$
$Q_{V^o}$ is $\frac{1}{2}$ times the adjacency matrix of the path  
$v_1 \sim v_2 \sim \dots \sim v_{M-2}$ of length $M-3$. 
$$
\spec(Q_{V^o}) = \Bigl\{ \tilde\lambda_{\ell} = \cos \frac{\ell\pi}{M-1}: \ell = 1, \dots, M-2 \Bigr\},
$$
disjoint from $\spec(Q_{V})$. We have $\nu_o(\tilde \lambda_\ell)=1$, and the normalised
eigenfunctions are 
$$
g_{\ell}(v_j) = 2\sin \frac{j\ell\pi}{M-1} \Big/ \sin \frac{\ell\pi}{M-1}\,,\quad j=1,\dots, M-2,
$$
Again, the type of $\tilde\lambda_\ell$ with respect to $Q_{V^o}$ is (II$^o$) when $\ell$ is odd and  
(III$^o$) when $\ell$ is even. 

The equations $\varphi(\lambda^*) = \lambda \in [-1\,,\,1]$ are not as straightforward
to solve as above. 

Suppose we have the host graph $X$ with its transition matrix $P$, leading to $X[V]$ and 
$P_*\,$. Let us count the eigenvalue multiplicities of $P_*$ which we have so far:
from each $\lambda_k \in \mathcal{S}\,$, $k=1,\dots, L-1$, we get $|X|$ linearly
independent eigenfunctions of $P_*\,$, and from each $\tilde \lambda_{\ell} \in \spec(Q_{V^o})$
we get $|E_X| - |X| + 1$, resp. $|E_X| - |X| + \delta_{\rm bi}$ linearly independent 
nodal eigenfunctions of
$P_*\,$, according to whether $\ell$ is odd or even. Together, the
resulting sum of multiplicities is
$$
|X|(L-1) + \bigl(|E_X| - |X|+1\bigr)(L-1) + \bigl(|E_X| - |X|+\delta_{\rm bi}\bigr)(L-1). 
$$
Now we observe that $\varphi(\tilde\lambda_{\ell}) = (-1)^{\ell}$ so that the equations
$\varphi(\lambda^*) = 1$ and $\varphi(\lambda^*) = -1$ only have the solutions 
$\lambda^* =1$, resp. $\lambda^* = -1$ which do not belong to $\spec(Q_{V^o})$.
Recall that $-1 \in \spec(P) \iff \delta_{\rm bi} = 1$. Thus, from those two equations we get
$1 + \delta_{\rm bi}$ solutions in $\mathcal{R}$, which is also the sum of their multiplicities
as eigenvalues of $P_*$.
The remaining number of equations $\varphi(\lambda^*) = \lambda \in \spec(P) \setminus \{\pm 1\}$
is $|X|-1-\delta_{\rm bi}\,$. Each of them transforms into the polynomial equation 
$\mathcal{T}_{L}(\lambda^*) = \lambda \mathcal{T}_{L-1}(\lambda^*)$ which has $L$ complex solutions,
a priori counting multiplicities. None of them can be in $\spec(Q_{V^o})$, since $\lambda \ne \pm 1$, 
and thus $\lambda^*$ is not a pole of $\psi(z)$.
By Corollary \ref{cor:allreal} each solution is in $(-1\,,\,1)$, and is simple by 
Lemma \ref{lem:psine0}.
Thus, each equation has exactly $L$ solutions which all belong to 
$(-1\,,\,1) \setminus \spec(Q_{V^o})$.
From these observations we obtain that
$$
\sum_{\lambda^* \in \mathcal{R}} \nu_*(\lambda^*) 
= 1 + \delta_{\rm bi} + \bigl(|X|-1-\delta_{\rm bi}\bigr)L\,.
$$
Again, the total sum of all multiplicities of the eigenvalues of $P_*$ is 
$|E_X|\bigl(|V|-2\bigr) + |X|$, as it has to be.
\end{exa}

\begin{exa}\label{exa:mple4} \emph{A concrete example where type \emph{(IV$^o$)} occurs.}
\\[4pt]
We take $V^o = \{ v_0\,,v_1\,,v_2\,,v_3\}$ as the 4-cycle, and $V = V^o \cup \{a,b\}$
with $a \sim v_0$ and $b \sim v_3\,$. Again, $Q$ is SRW. Computing $4 \times 4$-determinants,
$$
G_{V^o}(v_0,v_0|z) = G_{V^o}(v_3,v_3|z)= \frac{36z^2-15}{36z^3-25z} \AND 
G_{V^o}(v_0,v_3|z) = \frac{12}{36z^2-25} 
$$
From here, using \eqref{eq:psi-theta} and Lemma \ref{lem:Fab},
$$
\psi(z) = \frac{4}{36z^2-25}\,,\quad \theta(z) = \frac{12z^2-5}{36z^3-25z}
\AND \varphi(z) = \frac{36z^4 -37z^2 + 5}{4z}.
$$
We shall see from our final multiplicity check that we do not need to compute
$\spec(Q)$ and the types of its eigenvalues, i.e., they will not provide additional 
contributions to the spectrum and eigenfunction basis of $P_*\,$. In particular, 
$\mathcal{S}= \emptyset$. 
$$
\spec(Q_{V^o}) = \{ \tfrac56\text{\; [type (II$^o$)]}, 0\text{\; [type (IV$^o$)]},
-\tfrac56\text{\; [type (III$^o$)]}\}
$$
with
$\nu_o(\tfrac56) = \nu_o(-\tfrac56)=1$ and $\nu_o(0)=2$.
We suppose $(X,P)$ to be given.
According to the table of Proposition \ref{pro:mult}, 
$\lambda_0^* = \tfrac56$
contributes $|E_X|-|X|+\delta_{\rm bi}$ linearly independent eigenfunctions to $P_*$,
from $\lambda_1^*=0$ we get $2|E_X|-|X|$ and from $\lambda_2^*=-\tfrac56$ we get $|E_X|-|X|+1$.
Thus, the sum of the multiplicities of the nodal eigenvalues of $P_*$ 
is
$$
4|E_X| - 3|X| + 1 + \delta_{\rm bi}\,.
$$
We now come to the ``response'' set $\mathcal{R}$ and associated multiplicities.
A short discussion of the curve $z \mapsto \varphi(z)$ for $z \in [-1\,,\,1]$ shows that for 
any $\lambda \in [-1\,,\,1]$, the equation $\varphi(\lambda^*) = \lambda$ has $4$ distinct 
solutions $\lambda^* \in [-1\,,\,1]$. While $0$ is a pole of $\varphi(z)$, for the other two
eigenvalues of $Q_{V^o}$ we have $\varphi(\pm \tfrac56) = \mp 1$, to be excluded from $\mathcal{R}$.
Thus, each $\lambda \in \spec(P) \setminus \{ 1, -1\}$ leads to $4\nu_P(\lambda)$ linearly independent
eigenfuctions of $P^*$, while for $\lambda =1$, resp. $\lambda=-1$ when $X$ is bipartite,
we only get $3\nu_P(\lambda)$ linearly independent ``response'' eigenfunctions. 
The resulting sum of multiplicities is 
$$
4|X| - 1 - \delta_{\rm bi}\,.
$$
In total, the sum of the multiplicities obtained from $\mathcal{R}$ and $\mathcal{N}$ is 
$4|E_X| + |X| = |X[V]|$, which is the required number.
\end{exa}

\section{Final remarks and outlook}\label{sec:outlook}

(A) It is desirable to compute many further examples to complete the picture.
 
\medskip 

(B)
At first glance, one may believe that the same approach works for 
adjacency matrices  in the place of the stochastic ones considered here. However, in trying to 
answer a question by a referee, it turned out that this is not at all straightforward. An adaptation
to this setting remains reserved to future work.

\medskip

(C) A vertex $x$ of a graph is called spectrally dominant, if the local spectrum 
at $x$ (see Definition \ref{def:locspec}) is the entire spectrum.
In relation with the study of {\sc Bruchez,  de la Harpe and Nagnibeda}~\cite{BHN}, 
the construction presented here provides plenty of graphs which have 
spectrally non-dominant vertices.  
Namely, in almost all cases (and at least when $|X|$
has at least one even or two odd cycles), the vertices of $X$ are non-dominant in $X[V]$ -- the  
Green functions $G_*(x,x|z)$ at the vertices of $X \subset X[V]$ typically 
do not have the elements of $\spec(Q_{V^o})$ as poles. 

To be more precise, 
in \cite{BHN} the adjacency matrices are considered, while our results concern the normalised
version $Q$. However, Proposition \ref{pro:specQVo} applies equally to 
the restriction $A_{V^o}$
of $A$ to $V^o$, and the embedding Lemma \ref{lem:i} also remains unchanged.

In the papers of {\sc O'Rourke and Wood}~\cite{OW} and  {\sc Liu and Siemons}~\cite{LS}, 
the conjecture  is stated that ``most finite graphs'' have characteristic polynomial 
irreducible over $\Q$, and 
as observed in \cite{BHN}, this implies that all vertices are dominant. \cite{BHN} provides
counterxamples, and here we can quantify: into any finite graph $X$ we can 
substitute the graph $V$ of Figure~1. Both $Q_V$ and the adjacency matrix $A_V$ have an eigenvalue of 
type (I) $\equiv$ (I$^o$). One gets a counterexample with $|X| + 2|E_X|$ vertices and $5|E_X|$
edges. 

\medskip

(D) Of future interest (and subject of planned studies) is the situation when the host 
graph $X$ is \emph{infinite} and connected, and one considers 
the weighted $\ell^2$-spectrum, where the weights are the total conductances $m_{X[V]}(\cdot)$. 
A corresponding conjecture is stated in \cite{St}. 
Beyond this, one may also take $V$ to be infinite.

Another approach is to take for $V$ a -- say, compact -- Riemannian manifold with its
Laplace-Beltrami operator. It should of course have two designated suitable 
boundary pieces $a$ and $b$ and 
an isometry which exchanges the two. When substituting $V$ in the place of the edges, 
one gets a Laplacian on the Riemannian complex thus created. At its singular points, that
is, the boundary parts of $V$ glued at the vertices of $X$, one needs to require 
compatibility, most typically in terms of
a Kirchhoff condition involving the normal derivatives at the glued copies of $V$. The basic 
and well-known example arises when $V$ is a line segment with endpoints $a$ and $b$. 
This has been studied in great detail in the context of what are called quantum graphs 
or metric graphs.
We cite here two of the earlier papers concerning the relation with the spectrum of $X$,
{\sc Roth}~\cite{Ro} and {\sc Cattaneo}~\cite{Ca}. For many details -- mostly involving 
infinite graphs -- 
and an exhaustive reference list, we cite the recent monograph of {\sc Kostenko and 
Nicolussi}~\cite{KN}. Besides line segments, one may insert many different types
of manifolds, compare with {\sc Bendikov et al.}~\cite{BSSW}.

\medskip

(E) A particular instance of an infinite graph is obtained when one iterates the substitution
countably often and passes to the limit graph. In this case, it is desirable that the 
latter remains locally finite. Therefore one needs to restrict to graphs $V$ where the special
vertices $a$ and $b$ have degree $1$ (only one neighbour in $V$). 
The graphs obtained in this way will be  self-similar graphs such as those constructed 
by {\sc Kr\"on}~\cite{Kr} and by {\sc Malozemov and Teplyaev}~\cite{MT}.

\medskip

(F) Although we consider the terminology ``graph substitution'' most adequate,
one may also relate this construction and the associated spectral theory 
with the wide range of graph products. The classical ones are the Cartesian, direct and 
strong as well as the lexicographic product (as a non-commutative one). In addition,
there are several products in which at least one of the two partner graphs has a designated
root, such as the free product, wreath product, star product or the comb product.
An older reference is {\sc Schwenk}~\cite{Sch}, more recent selected references are the
monograph by {\sc Hora and Obata}~\cite{HO} and {\sc Arizmendi, Hasebe and  Lehner}~\cite{AHL}.
Besides wreath procucts -- see {\sc Grigorchuk and \.Zuk}~\cite{GZ}, {\sc Bartholdi 
and Woess}~\cite{BW} and {\sc Woess}~\cite{W} for specific spectral results for certain 
infinite graphs -- the comb product $X \triangleright V$ is maybe the closest one to substitution: 
if $V$ is a graph with root $o$ then one attaches a copy of $V$ by its root to every vertex of $X$.


\begin{thebibliography}{22}

\bibitem{AHL} Arizmendi, O., Hasebe, T., and  Lehner, F.: \emph{Cyclic 
independence: Boolean and monotone.} Algebr. Comb. {\bf 6} (2023), no. 6, 1697--1734.

\bibitem{BW} Bartholdi, L., and Woess, W.: \emph{Spectral computations on 
lamplighter groups and Diestel-Leader graphs.} J. Fourier Anal. Appl. {\bf 11} (2005) 175--202. 

\bibitem{BSSW} Bendikov, A., Saloff-Coste, L., Salvatori, M., Woess, W.:
\emph{The heat semigroup and Brownian motion on strip complexes.}
Adv. Math. {\bf 226} (2011), no. 1, 992–1055.

\bibitem{Bi} Biggs, N.: \emph{Algebraic Graph Theory.}
Cambridge Tracts in Math. {\bf 67}, 
Cambridge University Press, London, 1974. 

\bibitem{BHN} Bruchez, C., de la Harpe, P., and Nagnibeda, T.:
\emph{Spectral measures and dominant vertices in graphs of bounded degree.}
J. Operator Theory {\bf 92} (2024), no. 1, 215--256.

\bibitem{Ca} Cattaneo, C.: \emph{The spectrum of the continuous Laplacian 
on a graph.}
Monatsh. Math. 1{\bf 24} (1997), no. 3, 215--235.

\bibitem{Ch} Chung, F. R. K.: \emph{Spectral Graph Theory.}
CBMS Regional Conf. Ser. in Math. {\bf 92},
American Mathematical Society, Providence, RI, 1997. 

\bibitem{CdV} Colin de Verdi\`ere, Y.: \emph{Spectres de Graphes.}
Soc. Math. France, Paris, 1998.
 
\bibitem{CDS} Cvetkovi\'c, D. M., Doob, M., and Sachs, H.:
\emph{Spectra of Graphs. Theory and Application.}
Pure Appl. Math. {\bf 87},
Academic Press, New York--London, 1980.

\bibitem{GR} Godsil, Ch., and Royle, G.: \emph{Algebraic Graph Theory.}
Grad. Texts in Math. {\bf 207},
Springer-Verlag, New York, 2001. 

\bibitem{GZ} Grigorchuk, R. I., and \.Zuk, A.: \emph{The lamplighter group 
as a group generated by a 2-state automaton, and its spectrum. } Geom. Dedicata {\bf 87}
(2001) 209--244. 

\bibitem{HO} Hora, A., and Obata, N.: \emph{Quantum Probability and 
Spectral Analysis of Graphs.}
Theoret. Math. Phys., Springer-Verlag, Berlin, 2007.

\bibitem{HJ} Horn, R. A., and Johnson, Ch. R.: \emph{Matrix Analysis.}
Cambridge University Press, Cambridge, 1985. 

\bibitem{KN} Kostenko, A., and  Nicolussi, N.: \emph{Laplacians on Infinite Graphs.}
Mem. Eur. Math. Soc. {\bf 3}, EMS Press, Berlin, 2022.

\bibitem{Kr} Kr\"on, B.: \emph{Green functions on self-similar graphs and bounds 
for the spectrum of the Laplacian.}
Ann. Inst. Fourier (Grenoble) {\bf 52} (2002), no. 6, 1875--1900.

\bibitem{LS} Liu, F., and Siemons, J.: 
\emph{Unlocking the walk matrix of a graph.}
J. Algebraic Combin. {\bf 55} (2022), no. 3, 663--690.

\bibitem{LP} Lyons, R., and Peres, Y.: \emph{Probability on Trees and Networks.}
Camb. Ser. Stat. Probab. Math. {\bf 42},
Cambridge University Press, New York, 2016.

\bibitem{MT} Malozemov, L. and Teplyaev, A.: \emph{Self-similarity, operators and dynamics.}
Math. Phys. Anal. Geom. {\bf 6} (2003), no. 3, 201--218.


\bibitem{OW} O'Rourke, S., and Wood, Ph. M.:
\emph{Low-degree factors of random polynomials.}
J. Theoret. Probab. {\bf 32} (2019), no. 2, 1076--1104.

\bibitem{Ro} Roth, J-P.: \emph{Le spectre du Laplacien sur un graphe.}
In \emph{Th\'eorie du Potentiel}, pp. 521--539.
Lecture Notes in Math. {\bf 1096}, Springer-Verlag, Berlin, 1984.

\bibitem{Sch} Schwenk, A. J.: \emph{Computing the characteristic polynomial of a graph.}
in \emph{Graphs and Combinatorics,}  pp. 153--172,
Lecture Notes in Math. {\bf 406}, Springer-Verlag, Berlin-New York, 1974.

\bibitem{Se} Seneta, E.: \emph{Nonnegative Matrices and Markov Chains.}
Second edition, Springer Ser. Statist.,
Springer-Verlag, New York, 1981.

\bibitem{SY} Shi, Y. and Yu, Ch.: \emph{Comparisons of Dirichlet, Neumann and Laplacian 
eigenvalues on graphs and their applications.}
Calc. Var. Partial Differential Equations {\bf 64} (2025), no. 9, Paper No. 305.

\bibitem{St} Strichartz, R. S.: \emph{Transformation of spectra of graph 
Laplacians.} Rocky Mountain J. Math. {\bf 40} (2010), no. 6, 2037--2062.

\bibitem{W} Woess, W.: \emph{A note on the norms of transition operators 
on lamplighter graphs and groups.}
Internat. J. Algebra Comput. {\bf 15} (2005), no. 5-6, 1261--1272.

\bibitem{WMarkov} Woess, W.: \emph{Denumerable Markov Chains. Generating
Functions, Boundary Theory, Random Walks on Trees.} European Math.
Soc. Publishing House, 2009.  


\end{thebibliography}
\end{document}